\documentclass[11pt,onecolumn]{article}

\usepackage[margin=1.15in]{geometry}
\usepackage[T1]{fontenc}
\usepackage{lmodern}
\usepackage[final]{microtype}
\usepackage{titlesec}
\titleformat{\section}{\Large\bfseries}{\thesection}{0.75em}{}
\titleformat{\subsection}{\large\bfseries}{\thesubsection}{0.75em}{}
\titlespacing*{\section}{0pt}{2.4ex plus 0.8ex minus 0.2ex}{1.2ex plus 0.2ex}
\titlespacing*{\subsection}{0pt}{1.9ex plus 0.6ex minus 0.2ex}{0.9ex plus 0.2ex}
\usepackage{mathtools}
\allowdisplaybreaks
\usepackage{amsthm}
\usepackage{amssymb}
\usepackage[numbers,sort]{natbib}
\usepackage{authblk}
\usepackage{enumitem}
\usepackage{mathrsfs}
\usepackage{tikz-cd}

\setlist[enumerate,1]{label=(\alph*), leftmargin=2.2em, itemsep=0.2em, topsep=0.4em}
\setlist[enumerate,2]{label=(\roman*), leftmargin=2.2em, itemsep=0.15em, topsep=0.25em}

\usepackage{aliascnt}

\usepackage[pagebackref,colorlinks,
linkcolor={red!70!black},
citecolor={green!60!black},
urlcolor={blue!80!black}]{hyperref}
\usepackage[capitalize]{cleveref}

\newcommand{\R}{\mathbb{R}}
\newcommand{\N}{\mathbb{N}}
\newcommand{\Z}{\mathbb{Z}}
\newcommand{\Q}{\mathbb{Q}}
\newcommand{\C}{\mathbb{C}}
\newcommand{\K}{\mathbb{K}}
\newcommand{\id}{\mathsf{id}}
\newcommand{\ev}{\mathrm{ev}}

\newcommand{\cV}{\mathscr{V}}
\newcommand{\cW}{\mathscr{W}}
\newcommand{\cC}{\mathsf{C}}
\newcommand{\cD}{\mathsf{D}}
\newcommand{\cJ}{\mathsf{J}}
\newcommand{\op}{\mathrm{op}}
\newcommand{\name}[1]{\ulcorner #1 \urcorner}

\newcommand{\Ab}{\mathsf{Ab}}
\newcommand{\ROBan}{\mathsf{ROBan}}
\newcommand{\Ban}{\mathsf{Ban}}
\newcommand{\Mod}{\mathsf{Mod}}
\newcommand{\fdVect}{\mathsf{fdVect}}
\newcommand{\Vect}{\mathsf{Vect}}
\newcommand{\OrdVect}{\mathsf{OrdVect}}
\newcommand{\OpSys}{\mathsf{OpSys}}

\theoremstyle{plain}
\newtheorem{theorem}{Theorem}[subsection]
\newtheorem*{theorem*}{Theorem}
\newaliascnt{lemma}{theorem}
\newtheorem{lemma}[lemma]{Lemma}
\aliascntresetthe{lemma}
\newaliascnt{lem}{theorem}
\newtheorem{lem}[lem]{Lemma}
\aliascntresetthe{lem}
\newaliascnt{prop}{theorem}
\newtheorem{prop}[prop]{Proposition}
\aliascntresetthe{prop}
\newaliascnt{corollary}{theorem}
\newtheorem{corollary}[corollary]{Corollary}
\aliascntresetthe{corollary}
\newaliascnt{proposition}{theorem}
\newtheorem{proposition}[proposition]{Proposition}
\aliascntresetthe{proposition}
\theoremstyle{remark}
\newaliascnt{remark}{theorem}
\newtheorem{remark}[remark]{Remark}
\aliascntresetthe{remark}
\theoremstyle{definition}
\newaliascnt{definition}{theorem}
\newtheorem{definition}[definition]{Definition}
\aliascntresetthe{definition}
\newaliascnt{example}{theorem}
\newtheorem{example}[example]{Example}
\aliascntresetthe{example}
\crefname{theorem}{Theorem}{Theorems}
\crefname{lemma}{Lemma}{Lemmas}
\crefname{lem}{Lemma}{Lemmas}
\crefname{prop}{Proposition}{Propositions}
\crefname{proposition}{Proposition}{Propositions}
\crefname{corollary}{Corollary}{Corollaries}
\crefname{remark}{Remark}{Remarks}
\crefname{definition}{Definition}{Definitions}
\crefname{example}{Example}{Examples}

\title{Nonunital Operator Systems\\ as Modules in Enriched Category Theory}

\author{Tobias Fritz}
\author{Tim Netzer}
\affil{\small Department of Mathematics, University of Innsbruck, Austria}
\date{\today}
\begin{document}

\newgeometry{left=1.15in,right=1.15in,top=0.1in,bottom=1.15in}
\maketitle

\begin{abstract}
	An operator system is similar to a module over a ring, with the role of scalar multiplication played by the action of completely positive maps.
	Using enriched category theory, we make this analogy into a precise categorical equivalence,
	namely between a certain category of nonunital operator systems and a certain category of left modules over the category of matrix algebras
	enriched over regularly ordered Banach spaces.
	Using right modules instead yields an equivalence with a certain category of nonunital dual operator systems.

	We also develop general separation, representation and extension theorems for modules in enriched category theory.
	Specializing these to our nonunital operator systems recovers results which partly recover the corresponding classical theorems for operator systems.
\end{abstract}

\begingroup
\linespread{0.95}\selectfont
\setlength{\parskip}{0pt}
\vspace{0.5em}
\tableofcontents
\endgroup
\restoregeometry

\section{Introduction}

Unital operator systems are the natural linear setting for completely positive maps:
they retain the matrix order structure of $C^*$-algebras without requiring closure under multiplication~\cite{Paulsen}.
Thanks to this, they play a central role in noncommutative convexity theory~\cite{WebsterWinklerOperatorConvexity,KennedyKimManorNonunital}
and quantum information theory~\cite{GuptaMandayamSunderQuantumInformation,DeLasCuevasNetzerWonderland}.

An operator system $S$ is in particular a matrix ordered vector space, which means that every matrix level $M_n(S)$ comes equipped with a cone of positive elements $M_n(S)_+$,
which is such that the action of completely positive maps $M_n \to M_m$ preserves positivity.
One can view this action by completely positive maps as analogous to the action of scalars on a module over a ring.
Our first goal is to make this analogy more precise.
Doing so requires a framework flexible enough to accommodate both the matrix structure and a notion of positivity.
As we will see, this can be achieved within \textbf{enriched category theory},
and specifically by considering \textbf{modules} over enriched categories.
As a result, we obtain a categorical equivalence between a certain category of \textbf{nonunital operator systems} 
and a certain category of left modules over the category of matrix algebras enriched over regularly ordered Banach spaces.

A different categorical approach to operator systems was proposed in~\cite{bos}.
Both approaches encode an operator system $S$ essentially as a functor which assigns to every matrix algebra $M_n$ the space of hermitian elements $M_n(S)_h$
and to every completely positive map $M_n \to M_m$ the induced map $M_n(S)_h \to M_m(S)_h$, which is required to preserve positivity.\footnote{In~\cite{bos}, only compressions $x \mapsto \alpha^* x \alpha$ are considered explicitly, but the extension to arbitrary completely positive maps is immediate.}
An important difference arises in how the matrix levels are encoded:
the fact that the functor acts like $M_n \mapsto M_n(S)_h$ is a primitive requirement in~\cite{bos},
while it is a derived statement in our new approach.
As additional improvements, our setting treats infinite-dimensional operator systems with the same ease as finite-dimensional ones,
and it also naturally produces a certain notion of nonunital operator system.

Further features of the present work are as follows: 
\begin{itemize}
	\item Our approach recovers a certain nonunital version of operator systems, and does so without any finite-dimensionality requirement.
	\item The abstract categorical machinery produces several results that are purely categorical,
		but whose instantiations reproduce part of the fundamental theorems on operator systems.
			In particular, we recover the classical theorem of Choi--Effros on embedding unital operator systems into $\mathcal{B}(\mathcal{H})$~\cite[Theorem~13.1]{Paulsen}.
		We also prove a general categorical extension theorem, but in our current setting it does not yet recover Arveson's extension theorem for completely positive maps.
\end{itemize}

Of course, any more abstract approach to a well-established theory needs to at least reproduce the classical results in some form.
Perhaps most prominent among these are the following, where ``module'' means ``module over an enriched category'':
\begin{itemize}
	\item The fact that completely positive maps into matrix algebras separate points of an operator system (\cref{cor:op_sys_separation})
		will be recovered as an instance of a general categorical separation theorem for modules (\cref{separation}).
	\item Similarly, the Choi--Effros representation theorem for unital operator systems (\cref{rem:op_sys_representation})
		can largely be obtained from a general categorical representation theorem for modules (\cref{representation}).
		In addition, the actual instantiation of the abstract representation theorem amounts to a representation theorem for nonunital operator systems (\cref{cor:op_sys_representation}).
	\item We also provide a general categorical extension theorem for modules (\cref{extension}),
		which is of a similar flavor as Arveson's extension theorem for completely positive maps.
		However, in our current setting the former does not yet specialize to the latter.
\end{itemize}

We now present a summary of the paper by sketching the contents of each section.

\cref{sec:module_categories} provides some relevant background material on \textbf{left and right modules} over an arbitrary enriched category $\cC$
together with their morphisms.
After the main \cref{def:module}, we recall the enriched Yoneda lemma and Yoneda embedding in \cref{lem:enriched_yoneda}
and discuss representable modules $\cC(W, -)$ and $\cC(-, W)$.
Familiar examples such as modules over rings follow, as well as some observations on base change along monoidal functors (\cref{sec:base_change}).
Overall, this section mainly fixes the notation and conventions used later,
where operator systems appear as left modules and dual operator systems as right modules.

\Cref{sec:cosmos_modules} starts with the stronger assumption that the base of enrichment $\cV$ is a \textbf{B\'enabou cosmos} (\cref{def:benabou_cosmos}).
In this case, the categories of modules ${}_\cC\Mod$ and $\Mod_\cC$ are themselves canonically $\cV$-enriched.
We then recall additional categorical machinery that will be useful later on,
including the enriched co-Yoneda lemma (\cref{lem:enriched_co_yoneda}) and Morita equivalence.
We then fix a distinguished object $\bot$ in the cosmos,
which will play the role of ``dualizing object'' in the sense that we will be interested in modules of $\bot$-valued morphisms.
The resulting dualization functor ``$\neg$'' turns left modules into right modules and vice versa (\cref{lem:neg_op_module}),
and this produces in particular the crucial \textbf{corepresentable modules} $\neg \cC(-, W)$ and $\neg \cC(W, -)$.
We also comment on our use of logical symbols $\bot$ and $\neg$ in \cref{rem:negation}\ref{rem:negation_heyting}.

The central parts of the categorical machinery are presented in \cref{sec:regular_bigenerators,sec:main_results},
where we impose additional conditions on the distinguished object $\bot$, making it a \textbf{regular bigenerator} (\cref{def:regular_bigenerator}).
This definition can perhaps be thought of as an abstract categorical version of the Hahn--Banach theorem,
in the sense that a regular bigenerator $\bot$ enjoys similar properties as the ground field does in functional analysis.
The conceptual basis for our main categorical results is the enriched \textbf{dual Yoneda lemma} (\cref{prop:dual_yoneda_enriched}),
which identifies module morphisms into a corepresentable module $\mathcal{F} \to \neg \cC(W, -)$ with morphisms $\mathcal{F}(W) \to \bot$ in the underlying category $\cV_0$.
This then leads to our main categorical results:
\begin{enumerate}
	\item\label{item:separation}
			Our \textbf{separation theorem} asserts that corepresentable modules form a coseparating family whenever $\bot$ is a coseparator (\cref{separation}),
	\item The abstract \textbf{representation theorem} says that whenever $\bot$ is a regular bigenerator, then every module embeds into a product of corepresentable modules (\cref{representation}).
	\item\label{item:extension}
		The abstract \textbf{extension theorem} shows that if $\bot$ is injective in $\cV_0$ with respect to regular monomorphisms,
		then the same holds for every product of corepresentable modules in the module category,
\end{enumerate}
Thus in this abstract theory, products of corepresentable modules play the role of universal ``receiving'' objects somewhat analogous to the role of $\mathcal{B}(\mathcal{H})$ in the theory of operator systems.

As a warm-up application, \Cref{sec:banach_modules} specializes our categorical machinery to a theory of Banach modules.
In this case, the underlying category $\cV_0$ is the category of Banach spaces and linear contractions, and the distinguished object $\bot$ is the ground field $\K \in \{\R, \C\}$.
This is a regular bigenerator and regularly injective by the Hahn--Banach theorem, and dualization is ordinary Banach-space duality.
Consequently, the three abstract theorems yield the concrete separation, isometric representation, and extension statements for Banach modules collected in \cref{thm:banach_modules}.

Finally, \Cref{operator_systems} presents our main application to nonunital operator systems.
This is based on the B\'enabou cosmos $\ROBan$, which consists of \textbf{regularly ordered Banach spaces} and positive contractions (\cref{def:roban}).
\Cref{sec:regularly_ordered_banach_spaces} discusses its properties in some detail and proves that the ground field $\R$ is a regular bigenerator (\cref{prop:roban_regular_bigenerator}).
We subsequently introduce \textbf{operator systems} in \cref{def:os}: their Hermitian matrix levels are regularly ordered Banach spaces compatible with scalar conjugation and satisfying the \textbf{weak direct sum axiom} 
\[
	\left\|
	\begin{pmatrix}
		x & 0 \\
		0 & x
	\end{pmatrix}
	\right\|
	=
	\|x\|,
\]
which is a weaker version of Ruan's direct sum axiom for operator spaces.
We also introduce \textbf{dual operator systems} based on a suitably dual version of the weak direct sum axiom.
After introducing the $\ROBan$-category of matrix algebras and completely positive contractions $\mathsf{CMat}$,
we turn to our main structural result \cref{thm:op_sys_as_modules}:
\begin{theorem*}
	\begin{enumerate}
		\item The category of left $\mathsf{CMat}$-modules is equivalent to the category of operator systems.
		\item The category of right $\mathsf{CMat}$-modules is equivalent to the category of dual operator systems.
	\end{enumerate}
\end{theorem*}
We then instantiate the abstract categorical results~\ref{item:separation}--\ref{item:extension} in this case.
This yields in particular a representation theorem for operator systems, where the embedding is into a product of corepresentable modules (\cref{cor:op_sys_representation}).
We explain in \cref{rem:op_sys_representation} how this result reproduces the usual Choi--Effros embedding theorem for unital operator systems.
Unfortunately, our abstract extension theorem does not specialize to an extension theorem for operator systems in our sense,
because $\R$ fails to be regularly injective in $\ROBan$ (\cref{rem:op_sys_extension}).
Finally, \cref{rem:os_tensor_closed} gives a glimpse of how the general categorical machinery also induces a tensor product on operator systems and indicates how we expect this to relate to duality theory for operator systems.


We leave further applications of our general categorical results \cref{separation,representation,extension} to future work.

\section{Module categories of enriched categories}
\label{sec:module_categories}

\subsection{Basics}

Suppose that $\cV=(\cV_0,\otimes,I)$ is a monoidal category.
To simplify notation, we will often suppress associators and unitors in $\cV$ from the notation, which makes no difference since we can assume $\cV$ to be strict monoidal without loss of generality.

Then we consider small $\cV$-enriched categories $\cC$ and write $\cC_0$ for the underlying ordinary category,
which is obtained by applying the hom-functor $\cV_0(I, -)$ to the hom-objects of $\cC$.
As per standard conventions, this means that composition in $\cC$ is given by a family of morphisms
\[
	\circ \: : \: \cC(Y, Z) \,\otimes\, \cC(X, Y) \longrightarrow \cC(X, Z),
\]
where the unintuitive order of the hom-objects is chosen so as to match the order of composition in the ordinary category $\cC_0$, which is $(g, f) \mapsto g \circ f$.
The identity morphisms are given by a family of morphisms $j_X : I \to \cC(X, X)$.

\begin{definition}[{cf.~\cite[Section~3]{StreetEnrichedCohomology}}]
	\label{def:module}
	Let $\cC$ be a small $\cV$-category.
	\begin{enumerate}
		\item A \textbf{left $\cC$-module} $\mathcal{F}$ is given by a family of objects $\mathcal{F}(X) \in \cV_0$ for every $X \in \cC$ and a family of morphisms $\mathcal{F}_{X, Y} : \cC(X, Y) \otimes \mathcal{F}(X) \longrightarrow \mathcal{F}(Y)$
			satisfying associativity 
			\[
				\begin{tikzcd}
					\cC(Y, Z) \otimes \cC(X, Y) \otimes \mathcal{F}(X) \ar[r, "\circ \otimes \id"] \ar[d, "\id \otimes \mathcal{F}_{X, Y}"'] & \cC(X, Z) \otimes \mathcal{F}(X) \ar[d, "\mathcal{F}_{X, Z}"] \\
					\cC(Y, Z) \otimes \mathcal{F}(Y) \ar[r, "\mathcal{F}_{Y, Z}"] & \mathcal{F}(Z)
				\end{tikzcd}
			\]
			and unitality conditions
			\[
				\begin{tikzcd}
					\mathcal{F}(X) \ar[r, "\id"] \ar[d, "\cong"'] & \mathcal{F}(X) \\
					I \otimes \mathcal{F}(X) \ar[r, "j_X \otimes \id"] & \cC(X, X) \otimes \mathcal{F}(X) \ar[u, "\mathcal{F}_{X, X}"']
				\end{tikzcd}
			\]
		\item Given left $\cC$-modules $\mathcal{F}$ and $\mathcal{G}$, a \textbf{morphism of left $\cC$-modules} $\eta : \mathcal{F} \to \mathcal{G}$ is given by a family of morphisms $\eta_X : \mathcal{F}(X) \to \mathcal{G}(X)$ such that for every pair of objects $X, Y \in \cC$, the diagram
			\[
				\begin{tikzcd}
					\cC(X, Y) \otimes \mathcal{F}(X) \ar[r, "\mathcal{F}_{X, Y}"] \ar[d, "\id \otimes \eta_X"'] & \mathcal{F}(Y) \ar[d, "\eta_Y"] \\
					\cC(X, Y) \otimes \mathcal{G}(X) \ar[r, "\mathcal{G}_{X, Y}"] & \mathcal{G}(Y)
				\end{tikzcd}
			\]
			commutes.
	\end{enumerate}
	We denote the resulting ordinary category by ${}_\cC\Mod$.
\end{definition}

For every object $W \in \cC$, there is a \textbf{representable left module} $\cC(W, -)$ defined in the obvious way using composition in $\cC$ as the action.

A \textbf{right $\cC$-module} is defined similarly, but with action morphisms $\mathcal{F}_{X, Y} : \mathcal{F}(Y) \otimes \cC(X, Y) \longrightarrow \mathcal{F}(X)$,
and with the evident associativity and unitality axioms and the analogous naturality axiom for morphisms of modules.
We denote the resulting category of right $\cC$-modules by $\Mod_\cC$.
For every object $W \in \cC$, there is also a representable right $\cC$-module $\cC(-, W)$.

If the enriching category $\cV$ is symmetric monoidal, then every $\cV$-category $\cC$ has an opposite $\cV$-category $\cC^\op$ with the same objects and hom-objects $\cC^\op(X,Y) \coloneqq \cC(Y,X)$.
In this case, left $\cC$-modules are equivalently right $\cC^\op$-modules, since the symmetry identifies an action
\[
	\cC(X, Y) \otimes \mathcal{F}(X) \longrightarrow \mathcal{F}(Y)
\]
with an action
\[
	\mathcal{F}(X) \otimes \cC(X, Y) = \mathcal{F}(X) \otimes \cC^\op(Y, X) \longrightarrow \mathcal{F}(Y),
\]
which is precisely the structure of a right $\cC^\op$-module.
Thus we get an equivalence of categories ${}_\cC\Mod \cong \Mod_{\cC^\op}$.

\begin{example}
	\label{ex:ab_modules}
	We take $\cV_0 = \Ab$, equipped with the usual tensor product.
	Then we can consider a ring $R$ as an $\Ab$-category with a single object $\star$.
	A right module in the current sense is then precisely a right $R$-module in the usual sense, and our category of right modules is the usual category of right $R$-modules.
	A left $R$-module is equivalently a right $R^\op$-module.
\end{example}

\subsection{The Yoneda lemma and Yoneda embedding}

We first recall an enriched version of the Yoneda lemma and Yoneda embedding, see e.g.~\cite[\S 1.9]{Kelly}.
The variance matches standard conventions best if we work with right modules.

\begin{lem}
	\label{lem:enriched_yoneda}
	For every object $A \in \cC$ and every right $\cC$-module $\mathcal{F}$, there is a natural bijection
	\[
		\Mod_\cC(\cC(-, A), \mathcal{F}) \cong \cV_0(I, \mathcal{F}(A)).
	\]
	Moreover, the assignment $A \mapsto \cC(-, A)$ extends to a fully faithful embedding $\cC_0 \hookrightarrow \Mod_{\cC}$.
\end{lem}

We include the proof as a warm-up exercise in enriched category theory.

\begin{proof}
	Given a morphism of modules $\eta \colon \cC(-, A) \to \mathcal{F}$, we obtain a morphism in $\cV_0$ defined as the composite
	\[
		\Phi(\eta) \:\coloneqq\:
		\begin{tikzcd}[column sep=large, baseline=(current bounding box.center)]
			I \ar[r, "j_A"] & \cC(A, A) \ar[r, "\eta_A"] & \mathcal{F}(A).
		\end{tikzcd}
	\]
	Conversely, given $x \colon I \to \mathcal{F}(A)$, define for every $X \in \cC$ a morphism
	\[
		\Psi(x)_X \colon
		\begin{tikzcd}[column sep=large]
			\cC(X, A) \ar[r, "\cong"] & I \otimes \cC(X, A) \ar[r, "x \otimes \id"] & \mathcal{F}(A) \otimes \cC(X, A) \ar[r, "\mathcal{F}_{X, A}"] & \mathcal{F}(X).
		\end{tikzcd}
	\]
	We claim that these are the components of a morphism of $\cC$-modules $\Psi(x)$.
	This amounts to showing the commutativity of
	\[
		\begin{tikzcd}
			\cC(Y, A) \otimes \cC(X, Y) \ar[r, "\circ"] \ar[d, "\cong"] & \cC(X, A) \ar[d, "\cong"] \\
			I \otimes \cC(Y, A) \otimes \cC(X, Y) \ar[r, "\id \otimes \circ"] \ar[d, "x \otimes \id"]  & I \otimes \cC(X, A) \ar[d, "x \otimes \id"] \\
			\mathcal{F}(A) \otimes \cC(Y, A) \otimes \cC(X, Y) \ar[r, "\id \otimes \circ"] \ar[d, "\mathcal{F}_{Y,A} \otimes \id"] & \mathcal{F}(A) \otimes \cC(X, A) \ar[d, "\mathcal{F}_{X, A}"] \\
			\mathcal{F}(Y) \otimes \cC(X, Y) \ar[r, "\mathcal{F}_{X, Y}"] & \mathcal{F}(X)
		\end{tikzcd}
	\]
	which holds because each of the three squares commutes.

	We next show that $\Phi$ and $\Psi$ are inverse to each other.
	For $x \colon I \to \mathcal{F}(A)$, we have
	\[
		\begin{tikzcd}
			I \ar[r, "j_A"] \ar[d, "x"] & \cC(A, A) \ar[r, "\cong"] & I \otimes \cC(A, A) \ar[d, "x \otimes \id"] \\
			\mathcal{F}(A) \ar[dr, "\id"'] \ar[r, "\cong"]	     & \mathcal{F}(A) \otimes I \ar[r, "\id \otimes j_A"] & \mathcal{F}(A) \otimes \cC(A, A) \ar[dl, "\mathcal{F}_{A, A}"] \\
									     & \mathcal{F}(A) 
		\end{tikzcd}
	\]
	where the upper square commutes by the interchange law and the lower part by the unitality condition for the module $\mathcal{F}$.
	This shows that $\Phi(\Psi(x)) = x$, and hence $\Phi \circ \Psi = \id$.

	Conversely, let $\eta \colon \cC(-, A) \to \mathcal{F}$ be a morphism of modules.
	Then for each $X \in \cC$, the diagram
	\[
		\begin{tikzcd}
			& I \otimes \cC(X, A) \ar[dr, "\cong"] \ar[dl, "j_A \otimes \id"'] & \\
			\cC(A, A) \otimes \cC(X, A) \ar[rr, "\circ"] \ar[d, "\eta_A \otimes \id"'] & & \cC(X, A) \ar[d, "\eta_X"] \\
			\mathcal{F}(A) \otimes \cC(X, A) \ar[rr, "\mathcal{F}_{X, A}"] & & \mathcal{F}(X)
		\end{tikzcd}
	\]
	commutes, where the lower part is the morphism of modules square for $\eta$.
	This shows that $\Psi(\Phi(\eta))_X = \eta_X$.
	Thus $\Phi$ and $\Psi$ are inverse bijections, natural in both $A$ and $\mathcal{F}$.

	For the second statement,
	taking $\mathcal{F} = \cC(-, B)$ yields
	\[
		\Mod_\cC(\cC(-, A), \cC(-, B)) \cong \cV_0(I, \cC(A, B)) = \cC_0(A, B),
	\]
	which is exactly the required full faithfulness.
\end{proof}

\subsection{Base change}
\label{sec:base_change}

Let $\cW$ be another monoidal category, with underlying ordinary category $\cW_0$, and let $U : \cV_0 \to \cW_0$ be a lax monoidal functor.
Then $U$ induces a base change operation turning every $\cV$-enriched category $\cC$ into a $\cW$-enriched category $U_*(\cC)$ with the same objects and $U_*(\cC)(X, Y) \coloneqq U(\cC(X, Y))$.
For a given right $\cC$-module $\mathcal{F}$, we can also apply $U$ to the components $\mathcal{F}_{X,Y} : \mathcal{F}(Y) \otimes \cC(X, Y) \longrightarrow \mathcal{F}(X)$,
which in combination with the lax structure of $U$ induce a family of morphisms
\[
	U(\mathcal{F}(Y)) \otimes U(\cC(X, Y)) \longrightarrow U(\mathcal{F}(Y) \otimes \cC(X, Y)) \longrightarrow U(\mathcal{F}(X)).
\]
These define a right $U_*(\cC)$-module which we denote by $U_*(\mathcal{F})$.
In this way, we obtain a base change functor $U_* : \Mod_\cC \longrightarrow \Mod_{U_*(\cC)}$.

\begin{lem}
	\label{faithful}
	If $U$ is faithful, then so is $U_*$.
\end{lem}

\begin{proof}
	A morphism of right modules is given by its family of components.
	So if $\eta, \theta \colon \mathcal{F} \to \mathcal{G}$ satisfy $U_*(\eta) = U_*(\theta)$, then this means $U(\eta_X) = U(\theta_X)$
	for every object $X \in \cC$.
	By faithfulness of $U$, this implies $\eta_X = \theta_X$ for all $X$, and hence $\eta = \theta$.
\end{proof}

In terms of the paradigm of \emph{property, structure, stuff}~\cite[Section~2.4]{BS},
\cref{faithful} means that a $\cC$-module can be viewed as a $U_*(\cC)$-module equipped with extra structure.

\section{Module categories with enrichment in a B\'enabou cosmos}
\label{sec:cosmos_modules}

\subsection{Basics}
\label{rem:enriched_functor_category}

Suppose now that $\cV$ is symmetric monoidal closed, with tensor-hom adjunction
\begin{equation}
	\label{internalhom}
	\cV_0(A \otimes B, C) \cong \cV_0(A, B \Rightarrow C).
\end{equation}
Throughout the rest of the paper, we will write
\[
	\ev_{B, C} : (B \Rightarrow C) \otimes B \longrightarrow C
\]
for the associated evaluation morphism, which is the counit of the adjunction.
By the standard theory of adjunctions,
the evaluation morphism has the property that the bijections~\eqref{internalhom} are implemented by composing with the evaluation morphisms,
\begin{align}
	\label{internalhom_explicit}
	\begin{split}
		\cV_0(A, B \Rightarrow C) & \longrightarrow \cV_0(A \otimes B, C) \\
		f & \longmapsto \ev_{B,C} (f \otimes \id_B).
	\end{split}
\end{align}
Monoidal closedness also equips the objects of $\cV_0$ with a canonical $\cV$-category whose hom-objects are $A \Rightarrow B$.
Following~\cite[\S 1.6]{Kelly}, we denote this self-enriched category again by $\cV$ and identify its underlying ordinary category with $\cV_0$.

For example, applying the tensor-hom adjunction to \cref{def:module} shows that a left $\cC$-module $\mathcal{F}$ is equivalently given by a family of morphisms
\[
	\cC(X, Y) \longrightarrow \left( \mathcal{F}(X) \Rightarrow \mathcal{F}(Y) \right)
\]
satisfying evident compatibility conditions with identities and composition.
In other words, left $\cC$-modules are the same thing as $\cV$-enriched functors $\cC \to \cV$,
and morphisms of left $\cC$-modules are exactly enriched natural transformations.
Similarly, a right $\cC$-module $\mathcal{F}$ is equivalently given by morphisms
\[
	\cC(X, Y) \longrightarrow \left( \mathcal{F}(Y) \Rightarrow \mathcal{F}(X) \right)
\]
satisfying suitable conditions,
which amount to the data of a $\cV$-enriched functor $\cC^\op \longrightarrow \cV$.

\begin{definition}[{e.g.~\cite{StreetCosmoi}}]
	\label{def:benabou_cosmos}
	A \textbf{B\'enabou cosmos} is a symmetric monoidal closed category $\cV$ whose underlying ordinary category $\cV_0$ is complete and cocomplete.
\end{definition}

Under this additional assumption, we can also consider our categories of modules as the underlying ordinary categories of enriched functor categories.
Indeed thanks to the completeness assumption, for enriched functors $\mathcal{F},\mathcal{G}\colon \cC \to \cV$, there is a hom-object in $\cV_0$ given by the enriched end\footnote{We refer to \cref{app:limits_colimits} for the definition of enriched end and to \cite{Kelly} for a more in-depth development.}
\[
	[\cC, \cV](\mathcal{F}, \mathcal{G}) \coloneqq \int_{X \in \cC} \bigl(\mathcal{F}(X) \Rightarrow \mathcal{G}(X)\bigr).
\]
Likewise, for enriched functors $\mathcal{F},\mathcal{G}\colon \cC^\op \to \cV$, the hom-object in $[\cC^\op,\cV]$ is
\[
	[\cC^\op, \cV](\mathcal{F}, \mathcal{G}) \coloneqq \int_{X \in \cC^\op} \bigl(\mathcal{F}(X) \Rightarrow \mathcal{G}(X)\bigr).
\]
The ordinary hom-sets are then obtained by taking the underlying ordinary categories.
Thus when $\cV$ is a B\'enabou cosmos, we have equivalences of ordinary categories
\[
	{}_\cC\Mod \cong [\cC,\cV]_0, \qquad \Mod_\cC \cong [\cC^\op,\cV]_0.
\]

In addition to the Yoneda \cref{lem:enriched_yoneda},
we will also use the following form of the \textbf{co-Yoneda lemma}, also known as \textbf{Yoneda reduction}.

\begin{lemma}[{e.g.~\cite[Corollary~2.3]{DundasRondigsOstvaerEnrichedFunctors}}]
	\label{lem:enriched_co_yoneda}
	Suppose that $\cV$ is a B\'enabou cosmos and $\cC$ a small $\cV$-category.
	Then for every right $\cC$-module $\mathcal{F}$ and every $W \in \cC$, we have
	\begin{equation}
		\label{co_yoneda}
		\int^{X \in \cC}\mathcal{F}(X)\otimes\cC(W, X) \,\cong\, \mathcal{F}(W).
	\end{equation}
\end{lemma}

More precisely, the isomorphism~\eqref{co_yoneda} is implemented from left to right by the cowedge consisting of the action morphisms
\[
	\mathcal{F}_{W, X} \: : \:
	\mathcal{F}(X)\otimes\cC(W, X)\longrightarrow \mathcal{F}(W),
\]
and this concrete form is how we generally use the co-Yoneda lemma.

\begin{remark}
Throughout the remainder of this section,
we often phrase our main results without explicitly mentioning the morphism(s) which implement the statements that we prove.
However, what these morphisms are is often important both for later proofs and for the applications.
In this sense, our results display \emph{proof relevance}~\cite{nLabProofRelevance}.
\end{remark}

\begin{example}
	\label{ex:co_yoneda_ab}
	Continuing \cref{ex:ab_modules}, it is well-known that $\cV_0=\Ab$ underlies a B\'enabou cosmos,
	where a hom-object $A \Rightarrow B$ is given by the set of homomorphisms $\Ab(A, B)$ equipped with the pointwise group structure.
	As before, we take $\cC$ to be a ring $R$ considered as an $\Ab$-category with a single object $\star$.
	Then the resulting enrichment of $\Mod_R$ over $\Ab$ is given by the usual abelian group structure on the hom-sets of $R$-modules.

	For a right $R$-module $M$, the co-Yoneda lemma says that the action map induces the usual isomorphism
	\[
		M \otimes_R R \cong M,
		\qquad
		m \otimes r \longmapsto mr.
	\]
	Indeed, the coend
	\[
		\int^{\star \in \cC} M \otimes \cC(\star,\star)
	\]
	is precisely the quotient of the abelian group $M \otimes_{\Z} R$ by the balancing relations
	\[
		(mr)\otimes s = m \otimes rs,
	\]
	and this is exactly the construction of $M \otimes_R R$.
\end{example}

\begin{remark}
	\label{rem:day_convolution}
	Suppose that $\cC$ is a small \emph{symmetric monoidal} $\cV$-category, and write $\boxtimes$ and $E$ for its tensor product and unit.
	Then Day convolution equips the $\cV$-category of right modules $[\cC^\op,\cV]$ with a closed symmetric monoidal structure~\cite[Thm.~3.3 and \S 4]{DayClosedFunctorCategories}.
	Its tensor product is given by
	\[
		(\mathcal{F}\star\mathcal{G})(X)
		\coloneqq
		\int^{Y,Z\in\cC}
		\mathcal{F}(Y)\otimes\mathcal{G}(Z)\otimes\cC(X,Y\boxtimes Z),
	\]
	and its unit is the representable module $\cC(-,E)$.
	This coend exists thanks to the assumptions that $\cC$ is small and $\cV_0$ is cocomplete, and the resulting tensor product enjoys a nice universal property~\cite[Lem.~9.8.2]{RichterCategoriesHomotopy}.
	Completeness of $\cV_0$ similarly gives the ends defining the internal homs,
	\begin{align}
		\label{eq:day_convolution_internal_hom}
		\begin{split}
			(\mathcal{F}\Rightarrow\mathcal{G})(X)
			& =
			\int_{Y\in\cC} \bigl(\mathcal{F}(Y)\Rightarrow\mathcal{G}(X\boxtimes Y)\bigr) \\[3pt]
			& \cong
			[\cC^\op,\cV](\mathcal{F},\mathcal{G}(X\boxtimes -)).
		\end{split}
	\end{align}
	Naturally, all the same statements hold for left modules.
\end{remark}

\subsection{Morita equivalence and examples of enriched module categories}

This subsection is not essential for our main results.
It provides some relevant background material for the applications to operator systems in \cref{operator_systems} as well as some first examples for the general theory.

\begin{definition}[{e.g.~\cite[\S 5.5]{Kelly} and~\cite{LindnerMorita}}]
	Two small $\cV$-categories $\cC$ and $\cD$ are called \textbf{Morita equivalent} if their $\cV$-categories of right modules are equivalent:
	\[
		[\cC^\op,\cV] \simeq [\cD^\op,\cV].
	\]
\end{definition}

This implies an equivalence of the underlying ordinary categories $\Mod_\cC \simeq \Mod_\cD$, but is strictly stronger in general.
Under the usual smallness hypotheses, Morita equivalence admits a characterization in terms of the \textbf{Cauchy completion} $\widehat{\cC}$,
which can be defined, up to equivalence, as the full $\cV$-subcategory of $[\cC^\op,\cV]$ spanned by the closure of the representables under absolute weighted colimits~\cite[Thm.~2.13 and Cor.~2.14]{StewartCauchyCompletion}.\footnote{If this closure is not essentially small, then the same construction can be used in a larger universe, and the Cauchy-completion criterion below is then a statement in that larger setting.}
Then Morita equivalence is equivalent to the equivalence of Cauchy completions~\cite[Cor.~2.7]{StewartCauchyCompletion}:
\[
	[\cC^\op,\cV] \simeq [\cD^\op,\cV]
	\qquad\Longleftrightarrow\qquad
	\widehat{\cC} \simeq \widehat{\cD}.
\]
Since $\widehat{\cC}$ is its own Cauchy completion, it is Morita equivalent to $\cC$.
More precisely, restricting modules along the inclusion $\cC \hookrightarrow \widehat{\cC}$ induces an equivalence of $\cV$-categories
\[
	[\widehat{\cC}^\op,\cV] \simeq [\cC^\op,\cV].
\]

\begin{example}[See~\cite{Lawvere}]
	Let $\cV_0$ be Lawvere's category of extended non-negative reals, where objects are the numbers $[0, \infty]$, there is a unique morphism $r \to s$ if and only if $r \geq s$, and the monoidal structure is given by addition with $0$ as the monoidal unit.
	The closed structure is given by $r \Rightarrow s = |s - r|_+ \coloneqq \max\{0, s - r\}$ with the extra convention $\infty - \infty = 0$.
	This makes $\cV$ a B\'enabou cosmos.
	Then a $\cV$-category is a Lawvere metric space, i.e.~a set $X$ together with a function $d : X \times X \to [0, \infty]$
	satisfying the triangle inequality and $d(x, x) = 0$ for all $x \in X$.
	Its underlying ordinary category is the preordered set of points of $X$ ordered by $x \leq y$ if and only if $d(x, y) = 0$.

	A right $X$-module is then a function $F : X \to [0, \infty]$ satisfying the Lipschitz condition
	\[
		F(x) \leq F(y) + d(x, y) \qquad \forall x, y \in X.
	\]
	The category of right $X$-modules $\Mod_X$ is then the poset of such functions where an arrow $F \to G$ exists if and only if $F(x) \geq G(x)$ for all $x \in X$.
	Since $\cV$ is a B\'enabou cosmos, $\Mod_X$ can also be upgraded to the Lawvere metric space $[X^\op, \cV]$,
	whose hom-objects are given by the distances
	\[
		d(F, G) = \sup_{x \in X} |G(x) - F(x)|_+.
	\]
	If $X$ is a metric space in the usual sense, then the Cauchy completion $\widehat{X}$ of $X$ as a $\cV$-category is precisely the Cauchy completion of $X$ in the usual metric sense.
\end{example}

The following class of examples will be useful for our considerations on operator systems in \cref{operator_systems}.

\begin{remark}
\label{ex:mor}
	For general $\cV$ and $\cC$,
	we write $\cC_E$ for the full $\cV$-subcategory of $\cC$ on a single object $E$.
	Suppose that every other object of $\cC$ is an absolute weighted colimit of a functor landing in $\cC_E$.
	Then $\cC$ is Morita equivalent to $\cC_E$,
	with the equivalence given at the level of modules by
	\begin{align}
		\label{eq:morita_equivalence}
		\begin{split}
			[\cC^\op, \cV] & \stackrel{\simeq}{\longrightarrow} [\cC_E^\op, \cV] \\[2pt]
			\mathcal{F} & \longmapsto \mathcal{F}(E),
		\end{split}
	\end{align}
	where we slightly abuse notation by writing $\mathcal{F}(E)$ for the restriction of $\mathcal{F}$ to the single-object subcategory $\cC_E$.

	If, in addition, the unit morphism $j_E : I \to \cC(E,E)$ is an isomorphism, then we have
	\begin{equation}
		\label{eq:morita_equivalence_EE_is_I}
		[\cC^\op, \cV] \simeq
		[\cC_E^\op, \cV] \simeq \cV.
	\end{equation}
	Indeed a $\cC_E$-module is given by an object $V \in \cV_0$ together with a morphism
	\begin{equation}
		\label{eq:action}
		V \otimes \cC(E, E) \longrightarrow V
	\end{equation}
	satisfying the usual associativity and unitality conditions.
	But the unitality condition alone forces this morphism to be the one which makes the diagram
	\[
		\begin{tikzcd}[column sep=small]
			V \otimes \cC(E, E) \ar[rr] & & V \\
					& V \otimes I \ar[ul, "\id \otimes j_E"] \ar[ur, "\cong"'] &
		\end{tikzcd}
	\]
	commute.
	Thus under the identification of $\cC(E, E)$ with $I$ via $j_E$, the action morphism~\eqref{eq:action} is forced to be the coherence isomorphism $V \otimes I \cong V$.
	In this way, we obtain an equivalence of $\cV$-categories $[\cC_E^\op, \cV] \simeq \cV$.

	Still under the hypothesis that $j_E$ is an isomorphism, we can also describe the inverse equivalence explicitly.
	While the direction $[\cC^\op, \cV] \simeq [\cC_E^\op, \cV]$ is restriction as above,
	its inverse sends an object $V \in \cV_0$ to the right $\cC$-module
	\[
		V \otimes \cC(-, E),
	\]
	with action induced by composition in $\cC$ in the obvious way.
	To see that this is the inverse equivalence, it is enough to evaluate at $E$, where we indeed get $V \otimes \cC(E, E) \cong V$.

	Analogous statements apply to left modules, but with $\cC(-, E)$ replaced by $\cC(E, -)$.
\end{remark}

The following two examples prepare us for the applications to operator systems in \cref{operator_systems}, where similar phenomena will occur.

\begin{example}
	\label{ex:vect_modules}
	We continue \cref{ex:ab_modules,ex:co_yoneda_ab} with $\cV_0 = \Ab$.
	Then on single-object $\Ab$-categories, Morita equivalence reduces to the basic notion of Morita equivalence of rings.
	The Cauchy completion of $R$ as a single-object category is the $\Ab$-category of finitely generated projective right $R$-modules.
	Indeed it is a standard fact that two rings are Morita equivalent if and only if their categories of finitely generated projective modules are equivalent.


	Consider in particular the case that $R = K$ is a field.
	Then as an instance of \eqref{eq:morita_equivalence}, we obtain that the forgetful functor
	\begin{align*}
		\Mod_{\fdVect_K} & \longrightarrow \Vect_K \\[2pt]
		\mathcal{F} & \longmapsto \mathcal{F}(K)
	\end{align*}
	is an equivalence, where $\fdVect_K$ is the category of finite-dimensional vector spaces over $K$.\footnote{Technically, we need to assume a small version of the usual category of finite-dimensional vector spaces, as our $\cV$-category $\cC$ is assumed to be small throughout.}
	More generally, if $\cC \subseteq \fdVect_K$ is any full subcategory containing $K$, then its Cauchy completion is $\fdVect_K$ again,
	and hence the forgetful functor
	\begin{align*}
		\Mod_\cC & \longrightarrow \Vect_K \\[2pt]
		\mathcal{F} & \longmapsto \mathcal{F}(K)
	\end{align*}
	is an equivalence.
\end{example}

\begin{example}
	Let $\OrdVect$ be the category of ordered vector spaces, that is vector spaces $V$ over $\R$ equipped with a salient generating convex cone $V_+ \subseteq V$,
	where salient and generating mean, respectively,
	\[
		V_+ \cap (-V_+) = \{0\}, \qquad V = V_+ - V_+.
	\]
	Morphisms between ordered vector spaces are positive linear maps.
	Given ordered vector spaces $V$ and $W$, we equip the vector space $V \otimes W$ with the projective positive cone~\cite[1.12]{WittstockOrderedTensor},
	\[
		(V \otimes W)_+ \coloneqq \left\{ \sum_i v_i \otimes w_i : v_i \in V_+, \, w_i \in W_+ \right\},
	\]
	This projective cone is salient~\cite[Theorem~4.8]{WortelLexicographicCones}, and it is generating because $V_+$ and $W_+$ are generating.
	Equipped with this cone, $V \otimes W$ enjoys the usual universal property with respect to positive bilinear maps.
	Thus the obvious coherence isomorphisms indeed make $\OrdVect$ into a symmetric monoidal category.\footnote{Abstractly, the reason is that positive multilinear maps are the morphisms of a symmetric multicategory, and this multicategory is representable by the universal property of $V \otimes W$; see~\cite{HermidaRepresentableMulticategories} for representable multicategories.}
	The monoidal unit object is $\R$ with the usual positive cone $\R_+$.

	We next argue that $\OrdVect$ is monoidal closed.
	To this end, call a linear map $f \colon V \to W$ \textbf{regular} if it can be written as a difference of positive linear maps~\cite[\S 1.4]{AliprantisTourky}.
	Then we define $V \Rightarrow W$ to be the vector space of regular linear maps $V \to W$, equipped with the positive cone of positive linear maps.
	This cone is salient: if $f : V \to W$ is positive and such that $-f$ is also positive, then we must have $f(v) = 0$ for all $v \in V_+$ because $W_+$ is salient,
	and hence $f = 0$ since $V_+$ generates $V$.
	The tensor-hom adjunction
	\[
		\OrdVect(U \otimes V, W) \,\cong\, \OrdVect(U, V \Rightarrow W)
	\]
	follows from the universal property of the projective cone: positive maps $U \otimes V \to W$ are equivalently positive bilinear maps $U \times V \to W$, and these are equivalently positive maps $U \to (V \Rightarrow W)$, using that $U_+$ generates $U$.

	For completeness and cocompleteness, the required basic constructions are as follows.
	Products are computed on the underlying vector spaces with the componentwise cone.
	For equalizers, if $f,g \colon V \to W$ are positive maps, define
	\[
		E_+ \coloneqq \{v \in V_+ : f(v)=g(v)\}, \qquad E \coloneqq E_+ - E_+.
	\]
	Then $E$, equipped with positive cone $E_+$ and the evident inclusion $E \hookrightarrow V$, is the equalizer of $f$ and $g$.
	Indeed, every positive map into $V$ equalizing $f$ and $g$ sends positive elements into $E_+$, and hence factors uniquely through $E$ because the cone in its domain is generating.
	Coproducts are direct sums with the cone generated by the summand cones.
	For the coequalizer of positive maps $f,g \colon V \to W$, let
	\[
		q_0 \colon W \longrightarrow Q_0 \coloneqq W / \mathrm{span}\{f(v)-g(v):v\in V\}
	\]
	be the ordinary vector space quotient, and equip $Q_0$ with the image wedge $q_0(W_+)$.
	This wedge is generating but need not be salient, so one quotients further by its lineality space
	\[
		L \coloneqq q_0(W_+) \cap (-q_0(W_+)).
	\]
	Then $Q \coloneqq Q_0/L$, equipped with the image of $q_0(W_+)$, is an ordered vector space, and the composite $W \to Q_0/L$ is the coequalizer of $f$ and $g$.
	Indeed, any positive map $h \colon W \to Z$ with $hf = hg$ factors through $Q_0$; since $Z_+$ is salient, this factorization kills $L$,
	and therefore factors uniquely through $Q_0/L$.
	Since small limits and colimits can be constructed from products and equalizers, respectively from coproducts and coequalizers, $\OrdVect$ is complete and cocomplete.
	Thus $\OrdVect$ is a B\'enabou cosmos.

	For $\cC$, consider for example the full subcategory of $\OrdVect$ on the objects $\R^n$ with the usual positive cone $\R^n_+$,
	or equivalently the full subcategory consisting of all finite-dimensional ordered vector spaces with simplex cones.
	Then every object of $\cC$ is a finite biproduct of the monoidal unit $\R$, and biproducts are absolute colimits.
	Thus by~\eqref{eq:morita_equivalence_EE_is_I}, $\Mod_\cC$ is equivalent to $\OrdVect$ again.
\end{example}

\subsection{Dual modules} 

From now on, we assume that $\cV$ is a B\'enabou cosmos together with a chosen designated object, which we denote by $\bot \in \cV_0$.
Using the internal hom, this determines a contravariant functor
\begin{align*}
	\neg \: \colon \cV_0^\op & \longrightarrow \cV_0, \\
	A & \longmapsto A \Rightarrow \bot.
\end{align*}

\begin{remark}
	\label{rem:negation}
	\begin{enumerate}
		\item\label{rem:negation_heyting}
			Our notation of $\bot$ and $\neg$ is inspired by the following class of examples,
			which shows that the relation to intuitionistic logic goes beyond mere analogy.
			If $\cV_0$ is a thin category, that is a category where there is at most one morphism between any two objects,
			then $\cV_0$ is equivalently a preordered set, and the requirement for $\cV$ to be a B\'enabou cosmos with respect to meets $\land$ as the symmetric monoidal structure
			is equivalent to it being a complete Heyting algebra.
			Then the monoidal unit element is the top element $\top$.
			If we choose the bottom element for $\bot$, then $\neg A$ is the usual negation of $A$ in the Heyting algebra.
			Just like in our framework, double negations play a central role in the theory of Heyting algebras.
		\item\label{rem:negation_functor}
			The functor $\neg$ is in fact canonically a $\cV$-functor
				\[
					\begin{tikzcd}
						\cV^\op \ar[r, "\neg"] & \cV.
					\end{tikzcd}
				\]
				Indeed, the internal hom extends to a $\cV$-functor $\cV^\op \otimes \cV \to \cV$ which sends $(A,B)$ to $A\Rightarrow B$, and fixing the second variable at $\bot$ gives $\neg$; see~\cite[\S 1.6]{Kelly}.
				More concretely, the $\cV$ functor structure is given by the morphisms
				\[
					\begin{tikzcd}
						\left( B\Rightarrow A \right) \ar[r] & \left( \neg A\Rightarrow\neg B \right)
					\end{tikzcd}
				\]
				obtained by transposing the composition morphism
				\[
					\begin{tikzcd}[column sep=large]
						(B\Rightarrow A)\otimes\neg A
						\ar[r, "\cong"]
						&
						\neg A\otimes(B\Rightarrow A)
						\ar[r, "\circ"]
						&
						\neg B.
					\end{tikzcd}
				\]
				Thus, on underlying ordinary morphisms, this $\cV$-functor sends $f\colon B\to A$ to the precomposition morphism $\neg f\colon\neg A\to\neg B$ described above.
				Its compatibility with composition and identities follows from the corresponding properties of composition in the self-enrichment $\cV$.
			\item The $\cV$-functor $\neg$ is part of an enriched adjunction
				\begin{equation}
					\label{neg_enriched_adjunction}
					\neg^\op \dashv \neg.
				\end{equation}
				Indeed, the tensor-hom adjunction and the symmetry give isomorphisms
				\[
					\cV^\op(\neg A,B)
					=
					\bigl(B\Rightarrow\neg A\bigr)
					\cong
					\bigl(A\otimes B\Rightarrow\bot\bigr)
					\cong
					\bigl(A\Rightarrow\neg B\bigr)
					=
					\cV(A,\neg B)
				\]
				which are $\cV$-natural in $A$ and $B$, and hence exhibit the stated enriched adjunction~\cite[\S 1.11]{Kelly}.
				Upon restricting to the underlying ordinary categories, we in particular obtain that $\cV_0$ is a \emph{dialogue category} in the sense of Melli\`es~\cite[Section~4]{MelliesDialogue}.

			For every object $A \in \cV_0$, the identity morphism $\id_{\neg A} \in \cV_0(\neg A,\neg A)$ corresponds under the tensor-hom adjunction to the morphism
			\[
				\ev_{A,\bot} \colon \neg A \otimes A \to \bot,
			\]
			which we call the \textbf{pairing morphism}.
			Specializing~\eqref{internalhom_explicit} to this case means that the tensor-hom correspondence is implemented by
			\begin{align}
				\label{pairing_explicit}
				\begin{split}
					\cV_0(B,\neg A) & \longrightarrow \cV_0(B \otimes A,\bot), \\
					f & \longmapsto \ev_{A,\bot} (f \otimes \id_A).
				\end{split}
			\end{align}
			For any $\varphi : B \otimes A \to \bot$,
			we call its preimage on the left the \textbf{name} of $\varphi$ and denote it by\footnote{This is the terminology and notation of categorical logic~\cite[p.~165]{MacLaneMoerdijkSheaves}.}
			\[
				\name{\varphi} \: : \: B \longrightarrow \neg A.
			\]
			For example, the name of the pairing morphism $\ev_{A,\bot} : \neg A \otimes A \to \bot$ is by definition $\id_{\neg A}$.
			For $A = I$, we can choose $\neg I = \bot$ and assume without loss of generality that
			\begin{equation}
				\label{evI}
				\ev_{I,\bot} \: : \: \neg I \otimes I \longrightarrow \bot
			\end{equation}
			is the unitor $\bot \otimes I \cong \bot$, and we will make this assumption from now on.

			A first use case of pairing morphisms and names is that they let us describe the action of $\neg : \cV_0^\op \to \cV_0$ on morphisms.
			For $f : A \to B$, its dual $\neg f$ is by definition the unique morphism $\neg B \to \neg A$ such that the diagram
			\begin{equation}
				\label{negf}
				\begin{tikzcd}[column sep=large,row sep=large]
					\neg B \otimes A \ar[d, "\id_{\neg B} \otimes f"'] \ar[r, "\neg f \otimes \id_A"] & \neg A \otimes A \ar[d, "\ev_{A,\bot}"] \\
					\neg B \otimes B \ar[r, "\ev_{B,\bot}"] & \bot
				\end{tikzcd}
			\end{equation}
			commutes, or equivalently
			\[
				\neg f = \name{\ev_{B,\bot} (\id_{\neg B} \otimes f)}.
			\]

			On the other hand, if we first apply a braiding to get $A \otimes \neg A \to \bot$, 
			then the name of this morphism is a canonical morphism
				\begin{equation}
					\label{eta_double_dual}
					\eta_A \: : \: A \longrightarrow \neg\neg A,
				\end{equation}
				which is precisely the component at $A$ of the unit of the enriched adjunction~\eqref{neg_enriched_adjunction}, since both correspond to $\id_{\neg A}$ under the adjunction.
				In particular, $\eta$ is $\cV$-natural.
	\end{enumerate}
\end{remark}

\begin{example}
	Continuing from \cref{ex:ab_modules}, in $\Ab$ it is natural to work with $\bot = \Q/\Z$.
	Then $\neg A \coloneqq \Ab(A, \Q/\Z)$ is the usual dual group of any abelian group $A$.
\end{example}

\begin{lem}
	\label{lem:neg_colimits_to_limits}
	The contravariant functor $\neg \colon \cV_0^\op \to \cV_0$
	sends colimits in $\cV_0$ to limits in $\cV_0$.
\end{lem}

\begin{proof}
	The enriched adjunction~\eqref{neg_enriched_adjunction} restricts to an adjunction $\neg^\op \dashv \neg$ between the underlying ordinary categories.
	It follows that $\neg^\op$ preserves colimits.
	Reinterpreting this statement in $\cV_0$ rather than in $\cV_0^\op$, this means exactly that $\neg$ sends colimits in $\cV_0$ to limits in $\cV_0$.
\end{proof}

\begin{lem}
	\label{lem:neg_op_module}
	Let $\mathcal{F}$ be a left $\cC$-module.
	Then the pointwise dualization $(\neg \mathcal{F})(X) \coloneqq \neg\mathcal{F}(X)$
	carries a canonical right $\cC$-module structure.
\end{lem}

\begin{proof}
	By the discussion in \cref{rem:enriched_functor_category}, the left $\cC$-module $\mathcal{F}$ is equivalently a $\cV$-functor $\mathcal{F}\colon\cC\to\cV$.
	Passing to opposites and then composing with the $\cV$-functor $\neg$ from \cref{rem:negation}\ref{rem:negation_functor} gives a $\cV$-functor
	\[
		\begin{tikzcd}
			\cC^\op \ar[r, "\mathcal{F}^\op"] & \cV^\op \ar[r, "\neg"] & \cV.
		\end{tikzcd}
	\]
	Such a $\cV$-functor is precisely a right $\cC$-module, and its value at $X$ is $\neg\mathcal{F}(X)$.

	To describe its action concretely, write
	$\mathcal{F}_{X,Y}\colon\cC(X,Y)\otimes\mathcal{F}(X)\to\mathcal{F}(Y)$
	for the action of $\mathcal{F}$.
	Unwinding the enriched functor $\neg$, the action of $\neg\mathcal{F}$ is the morphism
	\[
		(\neg\mathcal{F})_{X,Y}\colon
		\neg\mathcal{F}(Y)\otimes\cC(X,Y)
		\longrightarrow
		\neg\mathcal{F}(X)
	\]
	characterized under the tensor-hom adjunction by the composite
	\[
		\begin{tikzcd}[column sep=large]
			\neg\mathcal{F}(Y)\otimes\cC(X,Y)\otimes\mathcal{F}(X)
			\ar[r, "\id\otimes\mathcal{F}_{X,Y}"]
			&
			\neg\mathcal{F}(Y)\otimes\mathcal{F}(Y)
			\ar[r, "\ev_{\mathcal{F}(Y),\bot}"]
			&
			\bot.
		\end{tikzcd}
	\]
	This is simply precomposition with the action of $\mathcal{F}$.
	Associativity and unitality now follow from the fact that the displayed composite of $\cV$-functors is again a $\cV$-functor.
\end{proof}

\begin{example}
	For $\cV_0 = \Ab$ and $\cC$ a ring $R$ considered as a one-object $\Ab$-category,
	the construction of \Cref{lem:neg_op_module} sends a left $R$-module $M$ to the right $R$-module
	\[
		\neg M \coloneqq \Ab(M, \Q/\Z),
	\]
	where the right action is defined by precomposition with the left action of $R$ on $M$,
	which we can write out as
	\begin{equation}
		\label{Ab_dual_module}
		(\varphi \cdot r)(m) \coloneqq \varphi(r \cdot m)
		\qquad \forall m \in M, r \in R, \varphi \in \neg M.
	\end{equation}
	This is a standard definition.
\end{example}

\begin{corollary}
	\label{cor:corepresentable_module}
	For every object $W \in \cC$, the pointwise dual $\neg \cC(W, -)$ is a right $\cC$-module, with action morphisms obtained by taking the composition morphisms
	\[
		\cC(Y,Z) \otimes \cC(W,Y)
		\longrightarrow
		\cC(W,Z)
	\]
	and dualizing $\cC(W,Z)$ to the left and then $\cC(W,Y)$ to the right.
\end{corollary}

We will call a module of this form \textbf{corepresentable}, where $W$ is the corepresenting object.

\begin{proof}
	The covariant representable $\cC(W,-)$ is a left $\cC$-module, with action morphisms given by composition in $\cC$.
	Hence the claim follows from \Cref{lem:neg_op_module} and the form of the action morphisms constructed in its proof.
\end{proof}

\begin{example}
	\label{ab_corep}
	Consider again $\Ab$ with $\bot = \Q/\Z$ and the one-object $\Ab$-category corresponding to a ring $R$.
	Then up to isomorphism, the only corepresentable right $R$-module is
	\[
		R^\vee \coloneqq \neg R = \Ab(R, \Q/\Z),
	\]
	where the right $R$-action is given by precomposition with the left action of $R$ on itself as in~\eqref{Ab_dual_module}.
\end{example}

\subsection{Regular bigenerators}
\label{sec:regular_bigenerators}

The following notion will play a key role in our main results.

\begin{definition}
	\label{def:regular_bigenerator}
	An object $\bot \in \cV_0$ is a \textbf{regular bigenerator} if the following two conditions hold for every $A \in \cV_0$:
	\begin{enumerate}
		\item the canonical morphism
			\begin{equation}
				\label{sharp_cond}
				\sharp_A \: \colon A \longrightarrow \prod_{\varphi \in \cV_0(A,\bot)} \bot,
			\end{equation}
			whose component at $\varphi$ is $\varphi \colon A \to \bot$ itself, is a regular monomorphism;
		\item the canonical morphism
			\begin{equation}
				\label{flat_cond}
				\flat_A \: \colon \coprod_{\varphi \in \cV_0(A,\bot)} I \longrightarrow \neg A,
			\end{equation}
			whose cocomponent at $\varphi$ is the name $\name{\varphi} \colon I \to \neg A$ of $\varphi$, is a regular epimorphism.
	\end{enumerate}
\end{definition}

\begin{example}
	\label{ab_cogen}
	In the case $\cV_0 = \Ab$, all monomorphisms and epimorphisms are regular because $\Ab$ is an abelian category.
	Then it is straightforward to see that the first condition for a regular bigenerator holds if and only if $\bot$ is a cogenerator~\cite[Section~3.3]{FreydAbelian}.
	The second condition holds for every abelian group as $\bot$: with $\neg A$ being the group of homomorphisms $A \to \bot$, we obtain that
	\[
		\flat_A \: \colon \coprod_{\varphi \in \Ab(A,\bot)} \Z \longrightarrow \neg A
	\]
	is the canonical surjection from the free abelian group generated by $\Ab(A,\bot)$ onto $\Ab(A,\bot)$ itself.
	Overall, we conclude that the regular bigenerators in $\Ab$ are precisely the cogenerators in the usual sense of an abelian category.
	The standard example is $\bot = \Q / \Z$.
\end{example}

Throughout the rest of this section, we will take $\cV$ to be a B\'enabou cosmos together with a distinguished object $\bot \in \cV_0$, which we often assume to be a regular bigenerator.

\begin{lem}
	\label{lem:neg_monos}
	For any object $A \in \cV_0$, consider the canonical morphism $\eta_A \colon A \to \neg\neg A$.
	\begin{enumerate}
		\item If $\bot$ is a coseparator, then $\eta_A$ is a monomorphism.
		\item If $\bot$ is a regular bigenerator, then $\eta_A$ is a regular monomorphism.
	\end{enumerate}
\end{lem}

\begin{proof}
	\begin{enumerate}
		\item To show that $\eta_A$ is monic, let
			$f,g \colon X \to A$
			be morphisms in $\cV_0$ such that $\eta_A f = \eta_A g$. We must prove that $f=g$.
			Since $\bot$ is a coseparator, it is enough to show that $\varphi f = \varphi g$ for every morphism
			$\varphi \colon A \to \bot$.
			So fix such a $\varphi$.

			Then we first note that the composite
			\[
				\begin{tikzcd}[column sep=large]
					A \ar[r, "\cong"] & A \otimes I \ar[r, "\id_A \otimes \name{\varphi}"] & A \otimes \neg A \ar[r, "\eta_A \otimes \id_{\neg A}"] & \neg\neg A \otimes \neg A \ar[r, "\ev_{\neg A,\bot}"] & \bot.
				\end{tikzcd}
			\]
			is $\varphi$ again.
			Indeed, by definition of $\eta_A$, the final two morphisms compose to the braiding followed by $\ev_{A,\bot}$,
			and hence the claim follows by definition of $\varphi$.

			Now precompose this expression for $\varphi$ with $f$.
			By naturality of the unitor, this gives
			\[
				\varphi f
				\;=\;
				\begin{tikzcd}[column sep=large]
					X \ar[r, "\cong"] & X \otimes I \ar[r, "f \otimes \name{\varphi}"] & A \otimes \neg A \ar[r, "\eta_A \otimes \id_{\neg A}"] & \neg\neg A \otimes \neg A \ar[r, "\ev_{\neg A,\bot}"] & \bot.
				\end{tikzcd}
			\]
			Similarly,
			\[
				\varphi g
				\;=\;
				\begin{tikzcd}[column sep=large]
					X \ar[r, "\cong"] & X \otimes I \ar[r, "g \otimes \name{\varphi}"] & A \otimes \neg A \ar[r, "\eta_A \otimes \id_{\neg A}"] & \neg\neg A \otimes \neg A \ar[r, "\ev_{\neg A,\bot}"] & \bot.
				\end{tikzcd}
			\]
			By the assumption $\eta_A f = \eta_A g$, these two composites are equal.
			Hence indeed $\varphi f = \varphi g$, as was to be shown.

		\item We assume now that $\bot$ is a regular bigenerator.
			For every $\varphi \colon A \to \bot$, let $\overline{\varphi} \colon \neg\neg A \to \bot$ be the composite
			\[
				\begin{tikzcd}[column sep=4pc]
					\neg\neg A \ar[r, "\cong"] & \neg\neg A \otimes I \ar[r, "\id_{\neg\neg A} \otimes \name{\varphi}"]
					& \neg\neg A \otimes \neg A \ar[r, "\ev_{\neg A,\bot}"] & \bot
				\end{tikzcd}
			\]
			By the computation in the previous part, we have $\varphi = \overline{\varphi}\eta_A$.
			To identify this morphism more explicitly, specialize~\eqref{negf} to
			$\name{\varphi} \colon I \to \neg A$:
			\[
				\begin{tikzcd}[column sep=large,row sep=large]
					\neg\neg A \otimes I
					\ar[r, "\neg\name{\varphi} \otimes \id_I"]
					\ar[d, "\id_{\neg\neg A} \otimes \name{\varphi}"']
					& \neg I \otimes I \ar[d, "\ev_{I,\bot}"] \\
					\neg\neg A \otimes \neg A
					\ar[r, "\ev_{\neg A,\bot}"']
					& \bot.
				\end{tikzcd}
			\]
			Upon precomposing both paths with the unitor
			$\neg\neg A \cong \neg\neg A \otimes I$, the lower path is
			$\overline{\varphi}$ by definition, while the upper path is
			$\neg\name{\varphi}$ because $\neg I = \bot$ and
			$\ev_{I,\bot}$ is the right unitor by~\eqref{evI}, using
			naturality of the right unitor.
			Hence
			\[
				\overline{\varphi} = \neg\name{\varphi}.
			\]
			Note also that the family $(\overline{\varphi})_{\varphi \in \cV_0(A,\bot)}$ induces a canonical morphism
			$\overline{\sharp}_A \colon \neg\neg A \to \prod_{\varphi \in \cV_0(A,\bot)} \bot$
			satisfying
			$\sharp_A = \overline{\sharp}_A \eta_A$.

			Let now
			$\flat_A \colon \coprod_{\varphi \in \cV_0(A,\bot)} I \to \neg A$
			be the canonical morphism from the definition of regular bigenerator.
			Since $\flat_A$ is a regular epimorphism and $\neg$ sends colimits to limits by \Cref{lem:neg_colimits_to_limits}, the morphism
			\[
				\neg \flat_A \: \colon \neg\neg A \longrightarrow \neg\!\left(\coprod_{\varphi \in \cV_0(A,\bot)} I\right)
			\]
			is a regular monomorphism.
			Under the canonical isomorphism
			\[
				\neg\!\left(\coprod_{\varphi \in \cV_0(A,\bot)} I\right) \cong \prod_{\varphi \in \cV_0(A,\bot)} \neg I \,\cong \prod_{\varphi \in \cV_0(A,\bot)} \bot,
			\]
			this morphism is exactly $\overline{\sharp}_A$, since its $\varphi$-component is $\overline{\varphi}$.
			In particular, $\overline{\sharp}_A$ is monic.

			By definition of regular bigenerator, $\sharp_A$ is a regular monomorphism.
			Consider the commutative square
			\[
				\begin{tikzcd}
					A \ar[r, "\eta_A"] \ar[d, equal] &
					\neg\neg A \ar[d, "\overline{\sharp}_A"] \\
					A \ar[r, "\sharp_A"] &
				\prod_{\varphi \in \cV_0(A,\bot)} \bot.
				\end{tikzcd}
			\]
			Since $\overline{\sharp}_A$ is monic and $\sharp_A = \overline{\sharp}_A \eta_A$, this square is a pullback.
			Regular monomorphisms are stable under pullback, so $\eta_A$ is a regular monomorphism.
			\qedhere
	\end{enumerate}
\end{proof}

The following is another manifestation of the observation that $\neg$ sends colimits to limits (\cref{lem:neg_colimits_to_limits}).

\begin{lem}
	\label{lem:neg_coends_to_ends}
	Let $\cJ$ be a small $\cV$-category and $H \colon \cJ^\op \otimes \cJ \to \cV$ a $\cV$-functor.
	Then there is a canonical isomorphism involving end and coend of $H$, namely
	\[
		\neg \int^{X \in \cJ} H(X,X)
		\cong
		\int_{X \in \cJ} \neg H(X,X).
	\]
\end{lem}

\begin{proof}
	Recall from \cref{def:enriched_coend} that the coend $\int^{X \in \cJ} H(X,X)$ is defined as the conical colimit in $\cV_0$ of
	the diagram containing each $H(X,X)$ and each $\cJ(Y,Z)\otimes H(Z,Y)$, with the latter mapping to the former by the two morphisms
	\[
		\cJ(Y,Z)\otimes H(Z,Y) \longrightarrow H(Z,Z),
		\qquad
		\cJ(Y,Z)\otimes H(Z,Y) \longrightarrow H(Y,Y)
	\]
	induced by the two actions of $\cJ(Y,Z)$ on the bifunctor $H$.
	Applying $\neg$ to this diagram yields a diagram on objects $\neg H(X,X)$ and
	\[
		\neg\bigl(\cJ(Y,Z)\otimes H(Z,Y)\bigr)
		\,\cong\,
		\cJ(Y,Z)\Rightarrow \neg H(Z,Y).
	\]
	Now let $K \colon \cJ^\op \otimes \cJ \to \cV$ be the bifunctor given by $K(Y,Z) \coloneqq \neg H(Z,Y)$.
	Then the right-hand side is exactly $\cJ(Y,Z)\Rightarrow K(Y,Z)$, so the objects of the pointwise dualized diagram agree with those in
	the corresponding diagram for enriched ends~\eqref{wedge_condition}.
	Moreover, the two morphisms
	\[
		\neg H(Z,Z) \longrightarrow \neg\bigl(\cJ(Y,Z)\otimes H(Z,Y)\bigr),
		\qquad
		\neg H(Y,Y) \longrightarrow \neg\bigl(\cJ(Y,Z)\otimes H(Z,Y)\bigr)
	\]
	obtained by dualizing the two displayed structure morphisms correspond under the same tensor-hom adjunction to
	\[
		K(Z,Z) \longrightarrow \cJ(Y,Z)\Rightarrow K(Y,Z),
		\qquad
		K(Y,Y) \longrightarrow \cJ(Y,Z)\Rightarrow K(Y,Z).
	\]
	The first is the arrow in~\eqref{wedge_condition} induced by the action of $K$ in its first variable, and the second is the one induced by the action of $K$ in its second variable.
	Hence the pointwise dualized diagram is canonically isomorphic to the conical limit diagram~\eqref{wedge_condition} computing the end of $K$.
	Applying \Cref{lem:neg_colimits_to_limits} to the colimit presentation of the coend, we thus obtain
	\[
		\neg \int^{X \in \cJ} H(X,X)
		\cong
		\int_{X \in \cJ} K(X,X)
		=
		\int_{X \in \cJ} \neg H(X,X),
	\]
	which is the desired isomorphism.
\end{proof}

\begin{lem}
	\label{lem:module_neg_monos}
	A morphism $\alpha \colon \mathcal{F} \to \mathcal{G}$ in $\Mod_\cC$ is a regular monomorphism if and only if every component $\alpha_X \colon \mathcal{F}(X) \to \mathcal{G}(X)$ is a regular monomorphism in $\cV_0$.
\end{lem}

\begin{proof}
	If $\alpha$ is a regular monomorphism in $\Mod_\cC$, then it is the equalizer of some parallel pair.
	Since equalizers in $\Mod_\cC$ are computed pointwise by \Cref{lem:module_pointwise_conical}, every component $\alpha_X$ is a regular monomorphism in $\cV_0$.
	Conversely, suppose that every $\alpha_X$ is a regular monomorphism.
	Then $\alpha_X$ is the equalizer of its cokernel pair, i.e.~of the two morphisms
	\[
		\mathcal{G}(X)
		\rightrightarrows
		\mathcal{G}(X) \amalg_{\mathcal{F}(X)} \mathcal{G}(X).
	\]
	By the pointwise computation of colimits, the self-pushout $\mathcal{G} \amalg_{\mathcal{F}} \mathcal{G}$ in $\Mod_\cC$ has precisely these components.
	Thus, again by the pointwise computation of equalizers, $\alpha$ is the equalizer of the two induced morphisms
	\[
		\mathcal{G}
		\rightrightarrows
		\mathcal{G} \amalg_{\mathcal{F}} \mathcal{G}.
	\]
	Hence $\alpha$ is a regular monomorphism in $\Mod_\cC$.
\end{proof}

\subsection{Main results}
\label{sec:main_results}

We continue assuming that $\cV$ is a B\'enabou cosmos together with designated object $\bot \in \cV_0$.
The following result is the core technical observation that will allow us to deduce general separation, representation and extension theorems for modules over $\cV$-categories.

\begin{prop}[Enriched dual Yoneda lemma]
	\label{prop:dual_yoneda_enriched}
	For every right $\cC$-module $\mathcal{F}$ and every $W \in \cC$, there is an isomorphism
	\begin{equation}
		\label{dual_yoneda}
		[\cC^\op,\cV](\mathcal{F}, \neg \cC(W,-))
		\,\cong\,
		\neg \mathcal{F}(W)
	\end{equation}
	natural in $\mathcal{F}$ and $W$.
\end{prop}

In particular, taking underlying hom-sets of~\eqref{dual_yoneda} gives a natural bijection
\begin{equation}
	\label{dual_yoneda_set}
	\Mod_\cC(\mathcal{F}, \neg \cC(W, -)) \cong \cV_0(\mathcal{F}(W), \bot).
\end{equation}

\begin{proof}
	We have
	\begin{align}
		\label{eq:dual_yoneda_proof}
		\begin{split}
			[\cC^\op,\cV](\mathcal{F}, \neg \cC(W,-))
			& =
			\int_{X \in \cC} \bigl(\mathcal{F}(X) \Rightarrow \neg \cC(W,X)\bigr) \\
			& \cong
			\int_{X \in \cC} \neg\bigl(\mathcal{F}(X)\otimes \cC(W,X)\bigr) \\
			& \cong
			\neg \int^{X \in \cC} \mathcal{F}(X)\otimes \cC(W,X) \\
			& \cong
			\neg \mathcal{F}(W),
		\end{split}
	\end{align}
	where the first equality is the definition of the hom-objects in $[\cC^\op,\cV]$,
	the second follows from the tensor-hom adjunction together with $\neg \cC(W,X) = \cC(W,X) \Rightarrow \bot$,
	the third is an instance of \Cref{lem:neg_coends_to_ends},
	and the fourth is by the enriched co-Yoneda \cref{lem:enriched_co_yoneda}.

	Finally, we identify the isomorphism more concretely by showing that it is the name of the composite
	\begin{equation}
		\label{dual_yoneda_explicit}
		\begin{tikzcd}[column sep=large]
			\left( \int_{X \in \cC} \mathcal{F}(X) \Rightarrow \neg \cC(W,X) \right) \otimes \mathcal{F}(W)
				\ar[d, "\pi_W\otimes\id"]
			\\
				\bigl(\mathcal{F}(W)\Rightarrow\neg\cC(W,W)\bigr)\otimes\mathcal{F}(W)
					\ar[r, "\ev_{\mathcal{F}(W),\neg\cC(W,W)}"]
			&
			\neg\cC(W,W)
				\ar[r, "\neg j_W"]
			&
			\bot
		\end{tikzcd}
	\end{equation}
	where $\pi$ is the universal wedge of the end, so that $\pi_W$ is one of its components.
	Indeed the co-Yoneda isomorphism
	\[
		\int^{X \in \cC}\mathcal{F}(X)\otimes \cC(W,X)
		\cong
		\mathcal{F}(W)
	\]
	is induced by the action maps $\mathcal{F}_{W,X}$, and its inverse is the composite
	\[
		\begin{tikzcd}[column sep=large]
			\mathcal{F}(W)
				\ar[r, "\cong"]
			&
			\mathcal{F}(W)\otimes I
				\ar[r, "\id\otimes j_W"]
			&
			\mathcal{F}(W)\otimes\cC(W,W)
				\ar[r]
			&
			\int^{X \in \cC} \mathcal{F}(X)\otimes \cC(W,X),
		\end{tikzcd}
	\]
	where the last arrow is the $W$-component of the universal cowedge.
	This is the inverse because composing with the action morphisms $\mathcal{F}_{W,X}$ gives the identity on $\mathcal{F}(W)$.
\end{proof}

\begin{lem}
	\label{lem:factor_corepresentable_embedding}
	Suppose that $\bot$ is a regular bigenerator.
	Then for every $A \in \cV_0$ and $X \in \cC$, there is a regular monomorphism
	\begin{equation}
		\label{corep_embedding}
		\left( \neg A \Rightarrow \neg \cC(X,-) \right)
		\:\longrightarrow
		\prod_{\varphi \in \cV_0(A,\bot)} \neg \cC(X,-)
	\end{equation}
	in $\Mod_\cC$.
\end{lem}

\begin{proof}
	We show that the morphism of type~\eqref{corep_embedding} whose $\varphi$-component at any $W \in \cC$ is evaluation at $\varphi$,
	i.e.~precomposition with the name $\name{\varphi} \colon I \to \neg A$, has regular monomorphism components.
	Its naturality together with~\cref{lem:module_neg_monos} then implies that we have a regular monomorphism in $\Mod_\cC$.

	Since $\bot$ is a regular bigenerator, the canonical morphism
	\[
		\flat_A \colon \coprod_{\varphi \in \cV_0(A,\bot)} I \longrightarrow \neg A
	\]
	is a regular epimorphism in $\cV_0$.
	For fixed $W$, applying the functor $(-)\Rightarrow\neg\cC(X, W)\colon\cV_0^\op\to\cV_0$ to $\flat_A$ yields a morphism
	\[
		\left( \neg A \Rightarrow \neg \cC(X,W) \right)
		\:\longrightarrow
		\left( \coprod_{\varphi \in \cV_0(A,\bot)} I \right) \Rightarrow \neg \cC(X,W),
	\]
	which is given by precomposition with $\flat_A$.
	Thanks to~\cref{lem:neg_colimits_to_limits} applied with $\neg \cC(X,W)$ as the distinguished object, 
	this morphism is a regular monomorphism in $\cV_0$.
	Another application of the same lemma lets us pull out the coproduct in its codomain, so that we obtain a regular monomorphism of type
	\begin{equation}
		\label{corep_embedding_pointwise}
		\left( \neg A \Rightarrow \neg \cC(X,W) \right)
		\:\longrightarrow
		\prod_{\varphi \in \cV_0(A,\bot)} \neg \cC(X,W).
	\end{equation}
	Its $\varphi$-component is precomposition with $\name{\varphi}$ by construction.
\end{proof}

We are now in a position to state and prove our main results.

\begin{theorem}[Separation theorem]
	\label{separation}
	Suppose that $\bot$ is a coseparator in $\cV_0$.
	Then the corepresentable modules are a coseparating family of objects in $\Mod_\cC$.
\end{theorem}

More explicitly, this means the following:
for any $\mathcal{F}, \mathcal{G} \in \Mod_\cC$ and $\alpha, \beta \colon \mathcal{F} \to \mathcal{G}$ with $\alpha \neq \beta$, there exist $A \in \cC$ and $\eta \colon \mathcal{G} \to \neg \cC(A, -)$ such that $\eta \alpha \neq \eta \beta$.

\begin{proof}
	Since $\alpha \neq \beta$, there is some $X \in \cC$ such that $\alpha_X \neq \beta_X$ as morphisms $\mathcal{F}(X) \to \mathcal{G}(X)$ in $\cV_0$.
	Because $\bot$ is a coseparator, there exists a morphism $\varphi \colon \mathcal{G}(X) \to \bot$
	with $\varphi \alpha_X \neq \varphi \beta_X$.
	By \Cref{prop:dual_yoneda_enriched}, this in turn corresponds to a morphism $\eta \colon \mathcal{G} \longrightarrow \neg \cC(X, -)$.
	Naturality of the dual Yoneda bijection in the module variable implies that $\eta\alpha$ corresponds to $\varphi\alpha_X$ and $\eta\beta$ corresponds to $\varphi\beta_X$.
	Since $\varphi \alpha_X \neq \varphi \beta_X$, it follows that $\eta \alpha \neq \eta \beta$.
\end{proof}

Let us say that a morphism of $\cC$-modules $\alpha \colon \mathcal{F} \to \mathcal{G}$ is an \textbf{embedding} if it is a composite of regular monomorphisms in $\Mod_\cC$.
Unfortunately, an embedding is not necessarily itself a regular monomorphism.\footnote{Following~\cite[p.~127]{KellyMonomorphismsPullbacks}, let $\Ab_{\setminus 4}$ be the category of abelian groups with no elements of order $4$, which is complete and cocomplete (and inherits the closed symmetric monoidal structure from $\Ab$). Then a monomorphism $f \colon A \to B$ in $\Ab_{\setminus 4}$ is regular if and only if $B/f(A)$ has no elements of order $4$. Thus the monomorphism $2 \colon \Z \to \Z$ is regular, while its composite with itself, which is $4 \colon \Z \to \Z$, is not regular. Now take $\cV_0 = \Ab_{\setminus 4}$ and let $\cC$ be the one-object $\cV$-category whose unique hom-object is $I$. Then a right $\cC$-module is just an object of $\cV$, since its only action morphism is necessarily the right unitor, and module morphisms are precisely morphisms in $\cV_0$. Hence $\Mod_\cC \cong \cV_0$, which gives the desired example.}
But since strong monomorphisms are closed under composition and every regular monomorphism is strong, we can at least say that every embedding is a strong monomorphism.

\begin{theorem}[Representation theorem]
	\label{representation}
	Let $\mathcal{F}$ be a right $\cC$-module.
	\begin{enumerate}
		\item\label{item:representation}
			Without any assumptions on $\bot$, there is a natural isomorphism
			\begin{equation}
				\label{eq:representation}
				\neg\neg \mathcal{F}
				\cong
				\int_{X \in \cC} \bigl(\neg \mathcal{F}(X) \Rightarrow \neg \cC(X,-)\bigr).
			\end{equation}

		\item Suppose that $\bot$ is a regular bigenerator.
			Then there is an embedding of $\mathcal{F}$ into a product of corepresentable modules of the form
			\begin{equation}
				\mathcal{F} \longrightarrow \prod_{X \in \cC} \prod_{\varphi \in \cV_0(\mathcal{F}(X),\bot)} \neg \cC(X,-).
			\end{equation}
	\end{enumerate}
\end{theorem}

\begin{proof}
	\begin{enumerate}
		\item
			$\neg \mathcal{F}$ is a left $\cC$-module by \Cref{lem:neg_op_module}.
			Now apply \Cref{prop:dual_yoneda_enriched} with $\cC^\op$ in place of $\cC$.
			For every object $W \in \cC$, this yields the natural isomorphism
			\[
				\int_{X \in \cC} \bigl(\neg \mathcal{F}(X) \Rightarrow \neg \cC(X,W)\bigr)
				\,=\,
				[\cC,\cV](\neg \mathcal{F}, \neg \cC(-,W))
				\,\cong\,
				\neg\neg \mathcal{F}(W),
			\]
			where the first step is just the definition of the hom-objects in $[\cC,\cV]$.

			Concretely, and described from right to left, the component of the isomorphism at any $W \in \cC$ is the name of the composite
			\[
				\begin{tikzcd}[column sep=large]
					\left(\int_{X \in \cC} \bigl(\neg \mathcal{F}(X) \Rightarrow \neg \cC(X,W)\bigr)\right)\otimes\neg\mathcal{F}(W)
					\ar[d, "\pi_W\otimes\id"]
					\\
					\bigl(\neg\mathcal{F}(W)\Rightarrow\neg\cC(W,W)\bigr)\otimes\neg\mathcal{F}(W)
					\ar[r, "\ev_{\neg\mathcal{F}(W),\neg\cC(W,W)}"]
						&
						\neg\cC(W,W)
						\ar[r, "\neg j_W"]
						&
						\bot
				\end{tikzcd}
			\]
			where $\pi_W$ is the $W$-component of the end.
			This explicit description follows by the one given in \Cref{prop:dual_yoneda_enriched}.
		\item
			Since $\eta$ is the enriched unit of~\eqref{neg_enriched_adjunction}, whiskering it with the $\cV$-functor $\mathcal{F}\colon\cC^\op\to\cV$ shows that the morphisms $\eta_{\mathcal{F}(X)}$ are the components of a morphism of modules
			\[
				\eta_{\mathcal{F}}\colon\mathcal{F}\longrightarrow\neg\neg\mathcal{F}.
			\]
			These components are regular monomorphisms by \Cref{lem:neg_monos}, and therefore $\eta_{\mathcal{F}}$ is a regular monomorphism in $\Mod_\cC$ by \Cref{lem:module_neg_monos}.
			It is therefore enough to prove the statement with $\neg\neg \mathcal{F}$ in place of $\mathcal{F}$.
			By~\eqref{eq:end_equalizer}, the end has a standard presentation as an equalizer of a parallel pair of morphisms out of the product
			\[
				\prod_{X \in \cC} \bigl(\neg \mathcal{F}(X) \Rightarrow \neg \cC(X,-)\bigr).
			\]
			Let
			\[
				e \: \colon \int_{X \in \cC} \bigl(\neg \mathcal{F}(X) \Rightarrow \neg \cC(X,-)\bigr)
				\longrightarrow
				\prod_{X \in \cC} \bigl(\neg \mathcal{F}(X) \Rightarrow \neg \cC(X,-)\bigr)
			\]
			denote this equalizer morphism, which is a regular monomorphism by definition of the latter.
			For each $X \in \cC$, let
			\[
				m_X \colon \bigl(\neg \mathcal{F}(X) \Rightarrow \neg \cC(X,-)\bigr)
				\longrightarrow
				\prod_{\varphi \in \cV_0(\mathcal{F}(X),\bot)} \neg \cC(X,-)
			\]
			be the regular monomorphism from \Cref{lem:factor_corepresentable_embedding}.
			Then
			\[
				\prod_X m_X \: \colon \prod_{X \in \cC} \bigl(\neg \mathcal{F}(X) \Rightarrow \neg \cC(X,-)\bigr)
				\longrightarrow
				\prod_{X \in \cC} \prod_{\varphi \in \cV_0(\mathcal{F}(X),\bot)} \neg \cC(X,-)
			\]
			is again a regular monomorphism by the commutation of equalizers with products.
			Overall, we thus obtain the desired embedding as the composite
			\[
				\begin{tikzcd}
					\mathcal{F}
					\ar[r, "\eta_{\mathcal{F}}"]
					&
					\neg\neg\mathcal{F}
					\ar[r, "\cong"]
					&
					\displaystyle \int_{X \in \cC} \bigl(\neg \mathcal{F}(X) \Rightarrow \neg \cC(X,-)\bigr)
					\ar[d, "e"] \\
					& & 
					\displaystyle \prod_{X \in \cC} \bigl(\neg \mathcal{F}(X) \Rightarrow \neg \cC(X,-)\bigr)
					\ar[r, "\prod_X m_X"]
					&
					\displaystyle \prod_{X \in \cC} \prod_{\varphi \in \cV_0(\mathcal{F}(X),\bot)} \neg \cC(X,-).
				\end{tikzcd}
			\]

			Finally, we show that the $(X,\varphi)$-component of our embedding at any $W \in \cC$ is the name of the composite
			\begin{equation}
				\label{eq:representation_explicit}
				\begin{tikzcd}
					\mathcal{F}(W) \otimes \cC(X, W) \ar[r, "\mathcal{F}_{X,W}"] & \mathcal{F}(X) \ar[r, "\varphi"] & \bot.
				\end{tikzcd}
			\end{equation}
			By~\eqref{dual_yoneda_set}, it is enough to identify the morphism $\mathcal{F}(X)\to\bot$ obtained by taking $W = X$ and precomposing with $\id_{\mathcal{F}(X)} \otimes j_X$.
			By the construction above and the description of $m_X$ as precomposition with $\name{\varphi}$ and naturality of dual Yoneda, this morphism is
			\[
				\neg\name{\varphi}\,\eta_{\mathcal{F}(X)}
				=
				\varphi,
			\]
			where the equality is the computation in the proof of \Cref{lem:neg_monos}.
			The concrete description~\eqref{dual_yoneda_explicit} of dual Yoneda together with the module unit axiom now shows that the $(X,\varphi)$-component at general $W$ is indeed precisely the name of~\eqref{eq:representation_explicit}.
			\qedhere
	\end{enumerate}
\end{proof}

Let us call an object $A$ in $\cV_0$ \textbf{regularly injective} if for every regular monomorphism $m : B \to C$ in $\cV_0$ and every morphism $f : B \to A$,
there is a morphism $g : C \to A$ which extends $f$ in the sense that the diagram
\[
	\begin{tikzcd}
		B \ar[rr, hook, "m"] \ar[dr, "f"'] & & C \ar[dl, dashed, "g"] \\
						   & A 
	\end{tikzcd}
\]
commutes.

\begin{theorem}[Extension theorem]
	\label{extension}
	Suppose that $\bot$ is regularly injective in $\cV_0$.
	Then products of corepresentable modules are regularly injective in $\Mod_{\cC}$.
\end{theorem}

\begin{proof}
	Since products of regularly injective objects are clearly regularly injective, it is enough to show that the corepresentable right module $\neg \cC(A,-)$ is regularly injective in $\Mod_{\cC}$, for $A \in \cC$.

	Let $m : \mathcal{F} \to \mathcal{G}$ be a regular monomorphism in $\Mod_{\cC}$ and $\alpha : \mathcal{F} \longrightarrow \neg \cC(A,-)$ a morphism of right $\cC$-modules.
	Under \eqref{dual_yoneda_set}, the morphism $\alpha$ corresponds to a morphism
	\[
		\widehat\alpha \: : \: \mathcal{F}(A) \longrightarrow \bot
	\]
	in $\cV_0$.
	Since the component $m_A : \mathcal{F}(A) \to \mathcal{G}(A)$ is a regular monomorphism in $\cV_0$ by \Cref{lem:module_neg_monos},
	we can use the regular injectivity of $\bot$ to get $\widehat\beta : \mathcal{G}(A) \to \bot$ which extends $\widehat\alpha$ along $m_A$.
	Applying the dual Yoneda isomorphism of \eqref{dual_yoneda_set} again, $\widehat\beta$ corresponds to a morphism of right modules
	\[
		\beta \: : \: \mathcal{G} \longrightarrow \neg \cC(A,-).
	\]
	Naturality of the dual Yoneda bijection with respect to $m$ implies that $\beta m$ corresponds to $\widehat\beta m_A = \widehat\alpha$,
	and therefore $\beta m = \alpha$.
	This proves that $\neg \cC(A,-)$ is regularly injective.
\end{proof}

\begin{example}
	Continuing on from~\cref{ab_corep}, we have the following:
	\begin{enumerate}
		\item The separation theorem states that every two distinct morphisms of right $R$-modules $\alpha, \beta \colon M \to N$ can be separated by an $R$-module homomorphism $N \to R^\vee$.
		\item The representation theorem states that every right $R$-module $M$ admits an embedding into a power of $R^\vee$, namely
			\begin{align*}
				M & \longrightarrow \prod_{\varphi \in \Ab(M,\Q/\Z)} R^\vee \\
				m & \longmapsto \left( a \mapsto \varphi(ma)) \right)_{\varphi \in \Ab(M,\Q/\Z)}.
			\end{align*}
			This concrete form is obtained from the general description in the proof at~\eqref{eq:representation_explicit}.
		\item The extension theorem states that $R^\vee$ is an injective right $R$-module:
			for every submodule inclusion $M \subseteq N$, every $R$-module homomorphism $M \to R^\vee$ can be extended to an $R$-module homomorphism $N \to R^\vee$.
			Under the adjunction $\Mod_R(-,R^\vee) \cong \Ab(-,\Q/\Z)$, this is precisely the usual injectivity of $\Q/\Z$ as an abelian group.
	\end{enumerate}
\end{example}

\section{Example: Banach modules}
\label{sec:banach_modules}

Let $\K \in \{\R,\C\}$.
In this section, we instantiate the preceding theory in the context of Banach modules over unital Banach algebras.
This is intended primarily as a warm-up example for the more involved case of operator systems considered in the next section.
We make no claims to originality of these results.
For background on Banach modules from a categorical perspective, we refer to~\cite{CiglerLosertMichor}.

\subsection{The category of Banach spaces}

\begin{definition}
	$\Ban$ denotes the category of Banach spaces over $\K$ and linear contractions, equipped with the projective tensor product as monoidal structure.
\end{definition}

The following observation is then well-known; see for example~\cite{CiglerLosertMichor}.

\begin{lem}
	$\Ban$ is a B\'enabou cosmos.
\end{lem}

\begin{proof}[Proof sketch]
	The tensor unit is $\K$, and the internal hom $V \Rightarrow W$ is the Banach space of bounded linear maps $V \to W$ equipped with the operator norm.
	Indeed, a contraction $U \to (V \Rightarrow W)$ is equivalently a contractive bilinear map $U \times V \to W$, and by the universal property of the projective tensor product this is equivalently a contraction $U \otimes_\pi V \to W$.

	Moreover, $\Ban$ is complete and cocomplete: products are given by $\ell^\infty$-sums, coproducts by $\ell^1$-sums,
	equalizers by closed subspaces, and coequalizers by the obvious Banach quotients.
\end{proof}

A $\Ban$-category with a single object $\star$ is precisely a monoid object in this symmetric monoidal category.
Equivalently, it is a unital Banach algebra $\cC(\star, \star) = \mathcal{A}$, with multiplication
corresponding to the composition map 
\begin{equation}
	\label{CstarA}
	\cC(\star, \star) \otimes_\pi \cC(\star, \star) \longrightarrow \cC(\star, \star),
\end{equation}
and unit corresponding to the identity morphism.
By the universal property of the projective tensor product, this is equivalently a contractive bilinear map $\mathcal{A} \times \mathcal{A} \to \mathcal{A}$ satisfying the same axioms, i.e.~it makes $\mathcal{A}$ into a unital Banach algebra.
A right $\cC$-module is then a Banach space $\mathcal{M}$ together with a contraction
\[
	\mathcal{M} \otimes_\pi \mathcal{A} \longrightarrow \mathcal{M},
\]
or in other words a contractive right Banach $\mathcal{A}$-module in the usual sense.
Thus in the one-object case, our notion of module recovers Banach modules~\cite[1.6]{CiglerLosertMichor}.\footnote{One caveat is that Banach modules in our sense automatically satisfy the condition that the unit acts trivially, which is not imposed in~\cite[1.6]{CiglerLosertMichor}.}

\begin{lem}[{e.g.~\cite[\S 4.3.10(e)]{Borceux}}]
	\begin{enumerate}
		\item The regular monomorphisms in $\Ban$ are precisely the isometric embeddings. 
		\item The regular epimorphisms in $\Ban$ are precisely the quotient morphisms, i.e.~the contractive surjections $q \colon X \to Y$ such that for every $y \in Y$ one has
			\[
				\|y\| = \inf \, \{\|x\| \mid q(x)=y\}.
			\]
			Equivalently, $q(X_{<1})=Y_{<1}$, where the subscript $1$ denotes the open unit ball.
	\end{enumerate}
\end{lem}

\begin{proof}
	\begin{enumerate}
		\item Every equalizer is the inclusion of a closed subspace with the induced norm.
			Conversely, every isometric embedding is the kernel of the quotient morphism onto the corresponding quotient space.
		\item Let $q \colon X \to Y$ be the coequalizer of two contractions $f,g \colon A \to X$.
			Then $q$ is obtained by quotienting $X$ by the closed linear subspace generated by the differences $f(a)-g(a)$.
			Hence $q$ is a contractive surjection and the norm on $Y$ is exactly the quotient norm, which gives the displayed formula.

			Conversely, suppose that $q \colon X \to Y$ is a quotient morphism in the stated sense.
			Let $K \coloneqq \ker(q)$, with inclusion $i \colon K \to X$.
			Then $q$ is the cokernel of $i$.
			\qedhere
	\end{enumerate}
\end{proof}

In particular, regular monomorphisms in $\Ban$ are closed under composition.

\begin{lem}
	Taking $\bot = \K$ gives a regular bigenerator in $\Ban$, and it is regularly injective.
\end{lem}

This makes $\neg V = (V \Rightarrow \K)$ the Banach space dual $V^*$.
As in \cref{ab_cogen}, the second property of a regular bigenerator is quite trivially true, while the first one is substantial and involves the Hahn--Banach separation theorem.

\begin{proof}
	Using the fact that products and coproducts in $\Ban$ are given by $\ell^\infty$-sums and $\ell^1$-sums, respectively,
	one finds that the canonical morphisms from the definition of regular bigenerator take the form
	\begin{align*}
		\sharp_V \: \colon V & \longrightarrow \ell^\infty(V^*_1) & \flat_V \: \colon \ell^1(V^*_1) & \longrightarrow V^* \\
		v & \longmapsto (\varphi(v))_{\varphi \in V^*_1} & (x_\varphi)_{\varphi \in V^*_1} & \longmapsto \sum_{\varphi \in V^*_1} x_\varphi \varphi.
	\end{align*}
	By Hahn--Banach, we have
	\[
		\|\sharp_V(v)\|_\infty = \sup_{\varphi \in V^*_1} |\varphi(v)| = \|v\|.
	\]
	for every $v \in V$.
	Thus $\sharp_V$ is an isometric embedding and hence a regular monomorphism.
	On the other hand, $\flat_V$ is a quotient morphism for the following reason.
	If $\psi \in V^*$ is nonzero, then $\psi/\|\psi\| \in V^*_1$ and we have
	\[
		\psi = \flat_V\bigl(\|\psi\| e_{\psi/\|\psi\|}\bigr)
	\]
	with $\|\|\psi\| e_{\psi/\|\psi\|}\|_1 = \|\psi\|$, while the case $\psi=0$ is trivial.
	Therefore every element of $V^*$ has a preimage of the same norm, and $\flat_V$ is a regular epimorphism.

	Finally, regular injectivity of $\K$ is precisely the statement of the Hahn--Banach extension theorem.
\end{proof}

With that in mind, \Cref{lem:neg_monos} now specializes to the fact that for every Banach space $V$,
the canonical morphism $V \to V^{**}$ is a regular monomorphism, i.e.~an isometric embedding.

\subsection{Banach modules}

Let now $\mathcal{A}$ be a unital Banach algebra, and $\cC$ the associated $\Ban$-category with a single object $\star$ via~\eqref{CstarA}.
Then there is only one corepresentable right $\mathcal{A}$-module up to isomorphism, namely the dual space $\mathcal{A}^*$ with $\mathcal{A}$-action
\begin{align*}
	\mathcal{A}^* \times \mathcal{A} & \longrightarrow \mathcal{A}^* \\
	(\varphi, a) & \longmapsto (x \mapsto \varphi(ax)).
\end{align*}
Then specializing \cref{separation,representation,extension} to this situation amounts to the following results.

\begin{theorem}
	\label{thm:banach_modules}
	\begin{enumerate}
		\item For any right Banach $\mathcal{A}$-modules $\mathcal{M}$ and $\mathcal{N}$ and any distinct contractive $\mathcal{A}$-module homomorphisms $\alpha,\beta \colon \mathcal{M} \to \mathcal{N}$, there exists a contractive $\mathcal{A}$-module homomorphism $\varphi \colon \mathcal{N} \to \mathcal{A}^*$ such that $\varphi\alpha \neq \varphi\beta$.
		\item For every right Banach $\mathcal{A}$-module $\mathcal{M}$, the canonical $\mathcal{A}$-module morphism
			\begin{equation}
				\label{representation_banach}
				\mathcal{M} \longrightarrow \prod_{\varphi \in \mathcal{M}_1^*} \mathcal{A}^*
			\end{equation}
			which sends $m \in \mathcal{M}$ to the family $(a \mapsto \varphi(ma))_{\varphi \in \mathcal{M}_1^*}$
			is an isometric embedding.
		\item If $\mathcal{N}$ is a right Banach $\mathcal{A}$-module and $\mathcal{M} \subseteq \mathcal{N}$ a closed submodule, then every contractive $\mathcal{A}$-module map $\mathcal{M} \to \mathcal{A}^*$ can be extended to a contractive $\mathcal{A}$-module map $\mathcal{N} \to \mathcal{A}^*$.
	\end{enumerate}
\end{theorem}

Note that the product in~\eqref{representation_banach} is an $\ell^\infty$-sum, meaning that it consists of uniformly bounded families $(\varphi_\psi)_{\psi \in \mathcal{M}_1^*}$ of elements of $\mathcal{A}^*$, with the norm given by $\|(\varphi_\psi)_\psi\| = \sup_\psi \|\varphi_\psi\|$.
\Cref{thm:banach_modules} is also not difficult to prove directly, essentially by instantiating our general proofs in the concrete setting of Banach modules.
In this sense, we regard it as a warm-up example for the more involved case of operator systems considered in the next section.

\section{Example: Operator systems}
\label{operator_systems}

Our goal here is to show how operator systems can be viewed as modules in the sense of the previous sections,
and how our \cref{separation,representation} apply in this case and recover standard theorems about operator systems.
We also explain why the hypothesis of our extension theorem \cref{extension} fails here.
In order to make operator systems fit the framework of modules, we need to go beyond the usual unital case and develop a theory of non-unital operator systems.
This will be based on regularly ordered Banach spaces, and has the appealing feature that the category of operator systems is closed under the formation of kernels and quotients,
and that duals are modules again, but over a slightly different enriched category.

Throughout this section, our ground field is $\R$ unless otherwise specified.
We will see that this is sufficient to even capture the usual operator systems over $\C$.

\subsection{The category of regularly ordered Banach spaces}
\label{sec:regularly_ordered_banach_spaces}

We start by introducing the category of regularly ordered Banach spaces studied by Min~\cite{MinExponentialLaw}.

\begin{definition}[{e.g.~\cite{MinExponentialLaw}}]
	A \textbf{regularly ordered Banach space} is a Banach space $V$ equipped with a closed convex cone $V_+ \subseteq V$ such that the following hold:
	\begin{enumerate}
		\item\label{item:regularly_ordered_1}
			If $-y \le x \le y$, then $\|x\| \le \|y\|$.
		\item\label{item:regularly_ordered_2}
			For every $x \in V$ with $\|x\| < 1$, there exists $y \in V_+$ with $\|y\| < 1$ and $-y \le x \le y$.
	\end{enumerate}
\end{definition}

Both conditions can also be combined into the single equation
\[
	\|x\| = \inf \, \{\|y\| \mid -y \le x \le y,\ y \in V_+\} \qquad \forall x \in V.
\]

\begin{example}
\label{ex:aous}
	We record that every Archimedean order unit space is a regularly ordered Banach space, as is every base norm space.\footnote{We refer to~\cite{Nagel} for the definitions. In our context, all spaces are additionally assumed to be complete and to have closed positive cone.}
	\begin{enumerate}
		\item Let $V$ be an order unit space with Archimedean order unit $u$ and order unit norm\footnote{The fact that this is a norm is how the Archimedeanicity is relevant.}
			\begin{equation}
				\label{eq:order_unit_norm}
				\|x\| \coloneqq \inf \, \{r > 0 \mid -ru \le x \le ru\}.
			\end{equation}
			Then it is easy to see that $-y \le x \le y$ implies $\|x\| \le \|y\|$.
			If $\|x\| < 1$, then by definition there exists $r < 1$ with $-ru \le x \le ru$, so taking $y = ru$ verifies the second condition.
			The Archimedeanicity is used to ensure that the order unit norm is indeed a norm.

			Conversely, if $V$ is a regularly ordered Banach space and $u \in V_+$ is any element satisfying~\eqref{eq:order_unit_norm}, then $u$ is an Archimedean order unit for $V$.
			Indeed the finiteness of the norm makes $u$ an order unit, and Archimedeanicity follows from closedness of $V_+$, since $x + \varepsilon u \in V_+$ for all $\varepsilon > 0$ implies $x = \lim_{\varepsilon \to 0} (x + \varepsilon u) \in V_+$.

			Moreover, it is not hard to see that the positive contractions between two Archimedean order unit spaces correspond precisely to the positive subunital linear maps.
		\item Let $V$ be a base norm space, so that $V_+$ admits a base defined by a strictly positive functional $\tau$ and the norm is given by
			\[
				\|x\| = \inf \, \{\tau(a) + \tau(b) \mid x = a - b,\ a,b \in V_+\}.
			\]
			On $V_+$ one has $\|z\| = \tau(z)$, hence the norm is additive on positive elements.
			If $-y \le x \le y$, then $\frac{y+x}{2},\frac{y-x}{2} \in V_+$, so
			\[
				\|x\| \le  \frac{1}{2}(\tau(y+x) + \tau(y-x)) = \tau(y) = \|y\|.
			\]
			If $\|x\| < 1$, choose $a,b \in V_+$ with $x = a - b$ and $\tau(a) + \tau(b) < 1$; then $y = a+b$ satisfies $-y \le x \le y$ and $\|y\| = \tau(a) + \tau(b) < 1$.
	\end{enumerate}
\end{example}

The following definition will provide the relevant monoidal structure on the category of regularly ordered Banach spaces.

\begin{definition}[{Wittstock~\cite{WittstockOrderedTensor}, Min~\cite[Section~3.1]{MinExponentialLaw}}]
	Let $V$ and $W$ be regularly ordered Banach spaces.
	Then their \textbf{projective tensor product} $V \otimes_\pi^+ W$ is defined as follows.
	\begin{enumerate}
		\item We let $(V \otimes_{\mathrm{alg}} W)_+$ be the cone generated by the simple tensors $v \otimes w$ with $v \in V_+$ and $w \in W_+$.
		\item For $x \in (V \otimes_{\mathrm{alg}} W)_+$, define the positive projective norm
			\[
				\|x\|_+ \coloneqq
				\inf
				\left\{
					\sum_i \|v_i\| \|w_i\|
					\,\middle|\,
					x = \sum_i v_i \otimes w_i,\ v_i \in V_+,\ w_i \in W_+
				\right\}.
			\]
		\item For arbitrary $x \in V \otimes_{\mathrm{alg}} W$, define
			\[
				\|x\| \coloneqq \inf \, \{ \|y\|_+ \mid -y \le x \le y,\ y \in (V \otimes_{\mathrm{alg}} W)_+ \}.
			\]
	\end{enumerate}
	The tensor product $V \otimes_\pi^+ W$ is then the Banach completion of $V \otimes_{\mathrm{alg}} W$ with respect to this norm and positive cone the closure of $(V \otimes_{\mathrm{alg}} W)_+$.
\end{definition}

This norm is (generally strictly) upper bounded by the usual Banach space projective tensor norm.\footnote{Given a decomposition $x=\sum_i v_i\otimes w_i$, regularity lets us choose $a_i\in V_+$ and $b_i\in W_+$ such that $-a_i\leq v_i\leq a_i$ and $-b_i\leq w_i\leq b_i$ and with $\|a_i\|$ and $\|b_i\|$ arbitrarily close to $\|v_i\|$ and $\|w_i\|$, respectively.
By positivity of the right-hand side, the identities
\[
	2(a_i\otimes b_i\pm v_i\otimes w_i)
	=
	(a_i+v_i)\otimes(b_i\pm w_i)
	+
	(a_i-v_i)\otimes(b_i\mp w_i)
\]
show that $-\sum_i a_i\otimes b_i\leq x\leq\sum_i a_i\otimes b_i$.
Consequently,
\[
	\|x\|
	\leq
	\sum_i\|a_i\|\|b_i\|,
\]
which is the relevant bound.}

\begin{definition}
	\label{def:roban}
	$\ROBan$ denotes the category of regularly ordered Banach spaces and positive linear contractions, equipped with $\otimes_\pi^+$ as symmetric monoidal structure.
\end{definition}

\begin{remark}
	\label{roban_misc}
	\begin{enumerate}
		\item We leave it understood that the associators, unitors, and symmetry are the obvious ones.
			The monoidal unit is $\R$ with its usual positive cone.
		\item\label{roban_morphisms} By~\cite[Proposition~1.1]{MinExponentialLaw}, a positive linear map $f \colon V \to W$ is already contractive as soon as it is contractive on positive elements.
			This implies that a linear map $f \colon V \to W$ is a morphism in $\ROBan$ if and only if it preserves positive subnormalized elements,
			meaning that
			\[
				x \ge 0, \quad \|x\| \le 1 \qquad \Longrightarrow \qquad f(x) \ge 0, \quad \|f(x)\| \le 1
			\]
			for all $x \in V$.
		\item\label{roban_universal_property}
			Essentially by definition, the projective tensor product $V \otimes_\pi^+ W$ enjoys the following universal property~\cite[Theorem~3.2]{MinExponentialLaw}:
			morphisms $V \otimes_\pi^+ W \to U$ to another regularly ordered Banach space $U$ are in natural bijection with those bilinear maps $f : V \times W \to U$
			which satisfy
			\begin{equation}
				\label{eq:roban_bilinear_morphism}
				\begin{aligned}
					x &\ge 0, & \|x\| &\le 1, \\
					y &\ge 0, & \|y\| &\le 1
				\end{aligned}
				\qquad \Longrightarrow \qquad f(x,y) \ge 0, \quad \|f(x,y)\| \le 1
			\end{equation}
			for all $x \in V$ and $y \in W$.
	\end{enumerate}
\end{remark}

\begin{remark}
	\label{roban_internal_hom}
	It was also shown by Min~\cite[Sections~4~and~5]{MinExponentialLaw} that $\ROBan$ is a B\'enabou cosmos.
	This is a consequence of the following facts.
	\begin{enumerate}
		\item Products are given by the usual (uniformly bounded) products of Banach spaces, equipped with the componentwise ordering.
		\item Coproducts are similarly given by $\ell^1$-sums.
		\item The description of equalizers, which is already more subtle, is given in \Cref{lem:roban_equalizer} below.
		\item Min only proves the existence of coequalizers by an abstract categorical argument without giving an explicit construction,
			and all that we will need in this direction is \cref{lem:roban_special_coequalizer},
			which gives a sufficient condition for a morphism to be a regular epimorphism.
		\item\label{roban_dual}
			Min discusses the construction of hom-objects in $\ROBan$~\cite[Section~2]{MinExponentialLaw}.
			In particular, for every $V \in \ROBan$ the object $V \Rightarrow \R$ has as underlying vector space the usual Banach dual $V^*$.
			Its positive cone consists of the positive linear functionals $V^*_+$, and its norm is
			\begin{equation}
				\label{eq:roban_dual_norm}
				\|\varphi\| = \inf \, \{ \|\psi\|_{\mathrm{op}} \mid \psi \in V^*_+,\, -\psi \le \varphi \le \psi \},
			\end{equation}
			where $\|\cdot\|_{\mathrm{op}}$ denotes the usual operator norm on $V^*$.
			The existence of some such $\psi$ follows by~\cite[Theorem~3.6.7]{JamesonOrderedLinearSpaces}.
			For positive $\varphi$, this norm coincides with the operator norm, but in general it is strictly larger.
	\end{enumerate}
\end{remark}

\begin{lem}[{Min~\cite[Proposition~5.2]{MinExponentialLaw}}]
	\label{lem:roban_equalizer}
	Let $f,g \colon V \to W$ be morphisms in $\ROBan$, and let
	\[
		D \coloneqq \{x \in V \mid f(x) = g(x)\},
		\qquad
		E \coloneqq (D \cap V_+) - (D \cap V_+).
	\]
	For $x \in E$, define
	\[
		\|x\|_E \coloneqq \inf \, \{ \|y\| \mid y \in D \cap V_+,\, -y \le x \le y \}.
	\]
	Then $E$, equipped with positive cone $D \cap V_+$ and this norm, is a regularly ordered Banach space, and the canonical injection
	\[
		e \colon E \longrightarrow V
	\]
	is an equalizer of $f$ and $g$ in $\ROBan$.
\end{lem}

Moreover, Min has shown that in general one may have $E \neq D$, and even when $E = D$, the norm $\|-\|_E$ need not coincide with the norm induced from $V$~\cite[Section~5.3]{MinExponentialLaw}.
Therefore regular monomorphisms in $\ROBan$ are not straightforward to characterize, but we have at least the following two partial observations.

\begin{lem}
	\label{roban_regular_embeddings}
	Every composite of regular monomorphisms in $\ROBan$ is an order embedding.
\end{lem}

\begin{proof}
	Every regular monomorphism in $\ROBan$ is an order embedding.
	Indeed, if $m \colon U \to V$ is a regular monomorphism, then by definition it is an equalizer of a parallel pair of morphisms out of $V$.
	By \Cref{lem:roban_equalizer}, such an equalizer is an order embedding.
	Since order embeddings are clearly closed under composition, the claim follows.
\end{proof}

\begin{lem}
	\label{roban_regular_embeddings2}
	Let $V$ be a regularly ordered Banach space and $U \subseteq V$ a closed subspace.
	If the restriction of the positive cone and norm of $V$ to $U$ makes $U$ into a regularly ordered Banach space,
	then the inclusion $U \hookrightarrow V$ is a regular monomorphism in $\ROBan$.
\end{lem}

\begin{proof}
	For every $x \in V \setminus U$, the Hahn--Banach theorem gives a bounded linear functional $\lambda_x \colon V \to \R$
	which vanishes on $U$ but not on $x$.
	Since $V \Rightarrow \R$ is regularly ordered by \cref{roban_internal_hom}, its positive cone is generating, and hence we can write
	\[
		\lambda_x = \lambda_x^+ - \lambda_x^-
	\]
	for positive bounded linear functionals $\lambda_x^+,\lambda_x^- \colon V \to \R$.
	After multiplying by the same sufficiently small positive scalar, we may assume that $\lambda_x^+$ and $\lambda_x^-$ are contractions.
	They therefore define morphisms
	\[
		r,s \: \colon \: V \longrightarrow \prod_{x \in V \setminus U} \R,
		\qquad
		r(v)_x \coloneqq \lambda_x^+(v),
		\qquad
		s(v)_x \coloneqq \lambda_x^-(v).
	\]
	Since each $\lambda_x$ vanishes on $U$, we have $r|_U=s|_U$.
	Conversely, if $x \notin U$, then $r(x)_x-s(x)_x\neq0$ by construction.
	Thus the set-theoretic equalizer of $r$ and $s$ is precisely $U$.

	By \Cref{lem:roban_equalizer}, the equalizer of $r$ and $s$ in $\ROBan$ has underlying vector space
	\[
		(U \cap V_+)-(U \cap V_+)=U_+-U_+=U,
	\]
	where the last equality follows from the fact that the positive cone of the regularly ordered space $U$ is generating.
	Moreover, its norm at every $x\in U$ is
	\[
		\inf\{\|y\| \mid y\in U_+,\ -y\le x\le y\}
		=
		\|x\|,
	\]
	which holds by the regularity assumption on $U$.
	Hence this equalizer is the given inclusion $U\hookrightarrow V$, which thus is a regular monomorphism.
\end{proof}

The following observation is largely due to Robinson and Yamamuro~\cite{RobinsonYamamuro}.

\begin{lem}
	\label{roban_half_norm}
	For $V \in \ROBan$, define the canonical half-norm
	\[
		N(a) \coloneqq \inf \, \{ \|b\| \mid b \in V_+,\, a \le b \}.
	\]
	Then
	\begin{equation}
		\label{half_norm2}
		N(a) = \sup_{\varphi \in \ROBan(V,\R)} \varphi(a)
	\end{equation}
	for every $a \in V$, and
	\[
		\max\{N(a), N(-a)\} \le \|a\| \le N(a) + N(-a).
	\]
	In particular, the norm $a \mapsto \max\{N(a),N(-a)\}$ is equivalent to the original norm.
\end{lem}

\begin{proof}
	We reduce the first claim to the results of Robinson and Yamamuro~\cite{RobinsonYamamuro}.
	For regularly ordered $V$, the ordinary operator norm on $V^*$ is $1$-monotonic.
	Indeed, suppose that $0\le\varphi\le\psi$ and $\|x\|<1$.
	By condition~\ref{item:regularly_ordered_2}, there exists $y\in V_+$ with $\|y\|<1$ and $-y\le x\le y$, and hence
	\[
		|\varphi(x)|\le\varphi(y)\le\psi(y)\le\|\psi\|_{\mathrm{op}}.
	\]
	Taking the supremum over all $x$ with $\|x\|<1$ gives $\|\varphi\|_{\mathrm{op}}\le\|\psi\|_{\mathrm{op}}$, which is the desired $1$-monotonicity.
	Therefore~\cite[Theorem~2.4]{RobinsonYamamuro} identifies their half-norm $x\mapsto\inf_{y\in V_+}\|x+y\|$ with our $N$,
	while \cite[Theorem~2.1]{RobinsonYamamuro} gives~\eqref{half_norm2}, since $\ROBan(V,\R)$ is the positive part of the dual unit ball.

	We now focus on the proof of the inequalities.
	For the first inequality, it is enough to show that $\|x\| < 1$ implies $N(x) < 1$.
	Hence assume $\|x\| < 1$.
	Then condition~\ref{item:regularly_ordered_2} gives $y \in V_+$ with $\|y\| < 1$ and $-y \le x \le y$.
	This witnesses $N(x) \le \|y\| < 1$.

	For the second inequality, let $\varepsilon > 0$.
	Choose $u,v \in V_+$ with
	\begin{align*}
		x & \le u, & N(x) & \ge \|u\| - \varepsilon, \\[2pt]
		-x & \le v, & N(-x) & \ge \|v\| - \varepsilon.
	\end{align*}
	Then combining the positivity and additional inequalities gives
	\[
		-(u+v) \le x \le u+v.
	\]
	Since $u+v \in V_+$, condition~\ref{item:regularly_ordered_1} yields
	\[
		\|x\| \le \|u+v\| \le \|u\| + \|v\| \le N(x) + N(-x) + 2\varepsilon.
	\]
	Letting $\varepsilon \to 0$, we obtain the claimed $\|x\| \le N(x) + N(-x)$.
	\qedhere
\end{proof}

While coequalizers in $\ROBan$ are elusive, we provide the following sufficient criterion.

\begin{lem}
	\label{lem:roban_special_coequalizer}
	Let $q \colon V \to W$ be a morphism in $\ROBan$ such that for every $y \in W_+$ there exists $x \in V_+$ with
	\[
		q(x)=y
		\qquad\text{and}\qquad
		\|x\|=\|y\|.
	\]
	Then $q$ is a regular epimorphism.
\end{lem}

\begin{proof}
	Consider the subspace of $V \times V$ defined by
	\[
		D \coloneqq \{(u,v) \in V \times V \mid q(u)=q(v)\},
	\]
	where $V \times V$ is equipped with the product ordering and norm.
	Then we first claim that
	\[
		D = \bigl(D \cap (V \times V)_+\bigr) - \bigl(D \cap (V \times V)_+\bigr).
	\]
	Indeed, let $(u,v) \in D$.
	Since $V_+$ is generating, we can write $u=u_1-u_2$ and $v=v_1-v_2$ with $u_i,v_i \in V_+$.
	Then
	\[
		(u,v)
		=
		(u_1+v_1,v_1+u_1) - (u_2+v_1,v_2+u_1),
	\]
	and both summands lie in $D \cap (V \times V)_+$.
	The first one does so trivially, and for the second we use
	\[
		q(u_2+v_1) = q(u_2)+q(v_1) = q(v_2)+q(u_1) = q(v_2+u_1),
	\]
	where the second equation follows from $q(u)=q(v)$.

	Therefore \Cref{lem:roban_equalizer} applied to the pair
	\[
		q\pi_1,\, q\pi_2 \colon V \times V \to W
	\]
	yields an equalizer $e \colon D \to V \times V$ whose underlying ordered vector space is precisely $D$.
	Consider now the two morphisms $D \to V$ given by
	\[
		p_1 \coloneqq \pi_1 e,
		\qquad
		p_2 \coloneqq \pi_2 e.
	\]
	We will prove that $q$ is the coequalizer of $p_1$ and $p_2$, which shows that $q$ is a regular epimorphism.

	To this end, let $h \colon V \to Z$ be any morphism in $\ROBan$ satisfying $hp_1 = hp_2$.
	Then $h(u)=h(v)$ whenever $q(u)=q(v)$, because every such pair $(u,v)$ comes from $D$.
	Now since $W_+$ is generating and $q$ is surjective onto $W_+$ by assumption, $q$ is also surjective on the whole underlying vector space $W$.
	We may thus define a linear map
	\[
		\overline{h} \: : \: W \longrightarrow Z
	\]
	by the rule $\overline{h}(q(x)) \coloneqq h(x)$.
	This is well-defined by what was already noted, and it satisfies $\overline{h}q=h$ by definition.
	Since the uniqueness of $\overline{h}$ is clear by surjectivity of $q$, it only remains to show that $\overline{h}$ is a morphism in $\ROBan$.

	To see that, let $y \in W_+$.
	Choose $x \in V_+$ with $q(x)=y$ and $\|x\|=\|y\|$.
	Then
	\[
		\overline{h}(y)=h(x) \in Z_+
	\]
	because $h$ is positive, and
	\[
		\|\overline{h}(y)\| = \|h(x)\| \le \|x\| = \|y\|.
	\]
	So $\overline{h}$ is a positive linear map that is contractive on the positive cone.
	By \Cref{roban_misc}\ref{roban_morphisms}, it is therefore a morphism in $\ROBan$.
\end{proof}

\begin{proposition}
	\label{prop:roban_regular_bigenerator}
	For every $V \in \ROBan$, and writing $V^* = (V \Rightarrow \R)$, we have:
	\begin{enumerate}
		\item The canonical morphism
			\[
				\sharp_V \colon V \longrightarrow \prod_{\varphi \in \ROBan(V,\R)} \R
			\]
			is a regular monomorphism.
		\item The canonical morphism
			\[
				\flat_V \colon \coprod_{\varphi \in \ROBan(V,\R)} \R \longrightarrow V^*
			\]
			is a regular epimorphism.
	\end{enumerate}
	Hence $\R$ is a regular bigenerator in $\ROBan$.
\end{proposition}

\begin{proof}
	\begin{enumerate}
		\item We have $\ROBan(V,\R) = V^*_{+, 1}$, the set of positive contractions $V \to \R$.
			Let then
			\[
				P \coloneqq \prod_{\varphi \in V^*_{+,1}} \R = \ell^\infty(V^*_{+,1}),
				\qquad
				\sharp \: : \: V \longrightarrow P,
			\]
			where we omit the subscript from $\sharp_V$ for brevity.
			We first show that the image of $\sharp$ is closed in $P$.
			For every $v \in V$, we have
			\[
				\|\sharp(v)\|_\infty
				=
				\sup_{\varphi \in V^*_{+,1}} |\varphi(v)|
				=
				\max\{N(v),N(-v)\}
			\]
			by \cref{roban_half_norm}.
			Thus also by \cref{roban_half_norm},
			the norm induced from $P$ on $\sharp(V)$ is equivalent to the given norm on $V$, and therefore $\sharp(V)$ is a closed subspace of the Banach space $P$.

			We next show that $\sharp$ reflects positivity, which makes it into an order embedding.
			First, $\sharp(v) \in P_+$ means that $\varphi(v) \ge 0$ for every positive contraction $\varphi \colon V \to \R$.
			If $v \notin V_+$, then the Hahn--Banach separation gives a positive linear functional $\varphi \in V^*_+$ with $\varphi(v) < 0$.
			After scaling, this yields a positive contraction $\varphi \colon V \to \R$ with $\varphi(v) < 0$, a contradiction.
			Thus $v \in V_+$, so $\sharp(V) \cap P_+ = \sharp(V_+)$.

			Now let
			\[
				J \coloneqq \{(u,v) \in \ROBan(P,\R)^2 \mid u \sharp = v \sharp \},
			\]
			and define morphisms
			\[
				r,s \colon P \longrightarrow \prod_{(u,v) \in J} \R
			\]
			by
			\[
				r(x)_{(u,v)} \coloneqq u(x),
				\qquad
				s(x)_{(u,v)} \coloneqq v(x).
			\]
			By construction one has $r \sharp = s \sharp$.
			We claim that $\sharp(V)$ is exactly the set-theoretic equalizer of $r$ and $s$.
			The inclusion $\sharp(V) \subseteq \operatorname{Eq}(r,s)$ is trivial by definition.
			Conversely, let $x \in P \setminus \sharp(V)$.
			Since $\sharp(V)$ is closed, the Hahn--Banach theorem gives a bounded linear functional $\lambda \in P^*$ with $\lambda|_{\sharp(V)} = 0$ and $\lambda(x) \neq 0$.
			Using \cite[Theorem~1.3]{RobinsonYamamuro}, we can write $\lambda = \lambda_+ - \lambda_-$ with $\lambda_+,\lambda_- \in P^*_+$.
			Choose $\alpha \ge \max\{\|\lambda_+\|,\|\lambda_-\|\}$ and set $u \coloneqq \alpha^{-1}\lambda_+$ and $v \coloneqq \alpha^{-1}\lambda_-$.
			Then $u,v \colon P \to \R$ are morphisms in $\ROBan$, and they satisfy $u \sharp = v \sharp$ because $\lambda$ vanishes on $\sharp(V)$.
			But $u(x) - v(x) = \alpha^{-1}\lambda(x) \neq 0$, so $r(x) \neq s(x)$.
			This proves $\operatorname{Eq}(r,s) = \sharp(V)$.

			Therefore by \Cref{lem:roban_equalizer}, the equalizer of $r$ and $s$ in $\ROBan$ is the vector space
			\[
				E \coloneqq (\sharp(V) \cap P_+) - (\sharp(V) \cap P_+)
			\]
			equipped with the norm
			\[
				\|x\|_E = \inf \, \{ \|y\|_\infty \mid y \in \sharp(V) \cap P_+,\, -y \le x \le y \}
			\]
			and the induced order.
			We now complete the proof by showing that $E = \sharp(V)$ and that $\|\sharp(v)\|_E = \|v\|$ for every $v \in V$,
			since this makes $\sharp$ into the equalizer of $r$ and $s$ in $\ROBan$.

			Since for every $v \in V$ there exists $w \in V_+$ with $-w \le v \le w$, we have
			\[
				\sharp(V) = (\sharp(V) \cap P_+) - (\sharp(V) \cap P_+) = E.
			\]
			Concerning the norms,
			it remains to be shown that $\|v\| = \|\sharp(v)\|_E$ for every $v \in V$.
			To this end, note first that
			\begin{align*}
				\|\sharp(v)\|_E
				& =
				\inf \, \{ \|\sharp(w)\|_\infty \mid w \in V_+,\, -w \le v \le w \} \\[2pt]
				& =
				\inf \, \{ \|w\| \mid w \in V_+,\, -w \le v \le w \} \\[2pt]
				& = \|v\|,
			\end{align*}
			where the first step is by the fact that $\sharp$ is an order embedding.
			For the second step, if $w\in V_+$, then $N(-w)=0$, so \Cref{roban_half_norm} gives
			\[
				\|\sharp(w)\|_\infty=\max\{N(w),N(-w)\}=\|w\|.
			\]
			The third step is the definition of regularly ordered Banach space.
		\item Let
			\[
				C \coloneqq \coprod_{\varphi \in V^*_{+,1}} \R = \ell^1(V^*_{+,1}),
				\qquad
				\flat \: : \: C \longrightarrow \neg V,
			\]
			where we omit the subscript from $\flat_V$ for brevity.
			By \cref{roban_internal_hom}, $V^*_+$ is the cone of positive bounded functionals on $V$,
			and the norm of every positive bounded functional is its operator norm.
			The map $\flat$ is induced by the family of names
			\[
				\name{\varphi} \: : \: \R \longrightarrow V^*
			\]
			for $\varphi \in V^*_{+,1}$, so it is equal to the linear map
			\begin{align*}
				\flat \: : \: \ell^1(V^*_{+,1}) & \longrightarrow V^* \\
				(x_\varphi)_{\varphi \in V^*_{+,1}} & \longmapsto \sum_{\varphi \in V^*_{+,1}} x_\varphi \varphi.
			\end{align*}
			We conduct the proof by showing that $\flat$ satisfies the assumptions of \Cref{lem:roban_special_coequalizer}.
			This makes $\flat$ into a regular epimorphism.

			So let $\psi \in V^*_+$.
			Since the case $\psi = 0$ is trivial, we may assume $\psi \neq 0$.
			Then $\psi / \|\psi\| \in V^*_{+,1}$, and
			\[
				x \coloneqq \|\psi\| e_{\psi / \|\psi\|} \in \ell^1(V^*_{+,1})_+
			\]
			satisfies
			\[
				\flat(x) = \psi
				\qquad\text{and}\qquad
				\|x\|_1 = \|\psi\|.
			\]
			Thus \Cref{lem:roban_special_coequalizer} shows that $\flat=\flat_V$ is a regular epimorphism.
			\qedhere
	\end{enumerate}
\end{proof}

\subsection{Operator systems as modules}

We now turn to our version of the category of (nonunital) operator systems, in which we will instantiate the general theory of \cref{sec:main_results}.

There are several possible variations on the definition of nonunital operator system, some of which have been considered in the literature.
We do not claim that our choice is the most natural or useful one and elaborate on this in \cref{rem:os_comments} below.
Thus the following definition and our subsequent developments should be understood as more of a proof of concept for the idea that ``operator systems are modules in enriched category theory'' rather than as a definitive treatment of nonunital operator systems.
For simplicity, we will omit the adjective ``nonunital'' in the following.

\begin{definition}
\label{def:os}
\label{def:dual_os}
\begin{enumerate}
	\item An \textbf{operator system} is a complex $*$-vector space $S$ equipped with a regularly ordered Banach space structure on the Hermitian part of every matrix level $M_n(S)$,
		and such that the following hold:
		\begin{itemize}
			\item For every $m,n \in \N$ and every scalar matrix $\alpha \in M_{n,m}$ with $\|\alpha\| \le 1$,
				the conjugation map
				\begin{align}
					\label{eq:conjugation_map}
					\begin{split}
						M_n(S)_h & \longrightarrow M_m(S)_h \\
						x & \longmapsto \alpha^* x \alpha
					\end{split}
				\end{align}
				is positive and contractive.
			\item The \textbf{weak direct sum axiom}
				\begin{equation}
					\label{eq:weak_direct_sum_axiom}
					\left\|
					\begin{pmatrix}
						x & 0 \\
						0 & x
					\end{pmatrix}
					\right\|
					=
					\|x\|
					\qquad
					\forall x \in M_n(S)_h
				\end{equation}
				holds at every matrix level $n$.
		\end{itemize}
	\item A \textbf{dual operator system} is defined the same way, but with the weak direct sum axiom replaced by the \textbf{dual weak direct sum axiom},
		which is the condition that the map
		\begin{align*}
			M_{2n}(T)_h & \longrightarrow M_n(T)_h \\
			\begin{pmatrix}
				x & z \\
				z^* & y
			\end{pmatrix}
				    & \longmapsto x+y
		\end{align*}
		is positive and contractive at every matrix level $n$.
	\item Given (dual) operator systems $S$ and $T$, a \textbf{morphism of (dual) operator systems} is a $*$-preserving linear map $f : S \to T$
		that is \textbf{completely positive} and \textbf{completely contractive},
		meaning that the induced map
		\[
			M_n(f)_h \: : \: M_n(S)_h \longrightarrow M_n(T)_h
		\]
		is positive and contractive for every $n \in \N$.
\end{enumerate}
\end{definition}

We thus obtain categories of operator systems $\OpSys$ and dual operator systems $\OpSys^\vee$.

\begin{remark}
	\label{rem:os_comments_2}
	\begin{enumerate}
		\item
			Abstract unital operator systems are usually defined with an order unit in $S$ that makes each matrix level into an Archimedean order unit space \cite{Paulsen}.
			Via \cref{ex:aous} and the usual matrix norm identities, these operator systems are also operator systems in our sense.
			Similarly, morphisms (in the sense of \cref{def:os}) between unital operator systems are completely positive subunital maps.
			Therefore the category of unital operator systems and completely positive subunital maps is a full subcategory of $\OpSys$.
		\item
			Our primary focus is on $\OpSys$.
			We consider $\OpSys^\vee$ mainly to illustrate that our general theory and the duality statement of \cref{lem:neg_op_module}
			yield a well-behaved duality theory for operator systems, where dualization is given by levelwise dualization of regularly ordered Banach spaces.
		\item\label{rem:os_wrong_dual}
			Of course, levelwise dualization of an operator system $S$ is not the only notion of dual that is of interest for operator systems:
			one often wants the dual at level $n$ to have as positive cone the space of completely positive maps $S \to M_n$.
			We comment on this more in \cref{rem:os_tensor_closed}.
	\end{enumerate}
\end{remark}

\begin{remark}
	\label{rem:os_comments}
	We make a few comments on our definition of operator system and how it relates to other definitions of nonunital operator systems in the literature.
	\begin{enumerate}
		\item It may seem strange to equip only the hermitian part of each matrix level with a norm.
			We have made this choice mainly for convenience, as it lets us exploit Min's results on the category $\ROBan$ of regularly ordered Banach spaces over $\R$,
			which we have made copious use of in the previous subsection.
			It seems plausible that one can prove analogous results also for \emph{complex} Banach spaces equipped with an involution,
			in which case we could develop the theory of this subsection with the norms defined on each entire $M_n(S)$.
			\item The norms on operator systems in our sense do not generally satisfy Ruan's direct sum axiom~\cite{RuanOperatorSpaces}, which would say
			\begin{equation}
				\label{eq:ruan_direct_sum_axiom}
				\left\|
				\begin{pmatrix}
					x & 0 \\
					0 & y
				\end{pmatrix}
				\right\|
				=
				\max\{\|x\|,\|y\|\}
			\end{equation}
			for all $x, y \in M_n(S)$, of which~\eqref{eq:weak_direct_sum_axiom} is merely the special case where $x=y$ and both are in the hermitian part.
			Concrete examples of this failure can be obtained by noting that our operator systems are closed under $\ell^1$-sums:
			if $S$ and $T$ are operator systems, then we define their \textbf{$\ell^1$-sum} $S \oplus_1 T$
			as the operator system whose underlying $*$-vector space is $S\oplus T$,
			and whose matrix levels
			\[
				M_n(S\oplus_1 T)_h
				=
				M_n(S)_h \oplus_1 M_n(T)_h,
			\]
			are equipped with the componentwise positive cone and the norm
			\[
				\|(x,y)\| \coloneqq \|x\|+\|y\|.
			\]
			This is again an operator system:
			each matrix level is regularly ordered because this $\ell^1$-sum is the coproduct in $\ROBan$,
			while the other two conditions are straightforward to check.

			In this new operator system, we have
			\[
				\left\|
				\begin{pmatrix}
					(x,0) & (0,0) \\
					(0,0) & (0,y)
				\end{pmatrix}
				\right\|
				=
				\left\|
				\begin{pmatrix}
					x & 0 \\
					0 & 0
				\end{pmatrix}
				\right\|
				+
				\left\|
				\begin{pmatrix}
					0 & 0 \\
					0 & y
				\end{pmatrix}
				\right\|
				= \|x\| + \|y\|,
			\]
			while
			\[
				\max\{\|(x,0)\|,\|(0,y)\|\} = \max\{\|x\|,\|y\|\},
			\]
			so that~\eqref{eq:ruan_direct_sum_axiom} fails unless $x=0$ or $y=0$.

			The full direct sum axiom~\eqref{eq:ruan_direct_sum_axiom} does not look like something that can be expressed in terms of the enriched category theory of modules.
			However, it can plausibly be captured by the more expressive framework of \emph{enriched algebraic theories}~\cite{LucyshynWrightEnrichedAlgebraicTheories}.
			Along these lines, we also conjecture that there is an enriched algebraic theory---with enrichment in the category of complex Banach spaces---whose models are precisely operator spaces.
	\end{enumerate}
	There are several other definitions of nonunital operator systems in the literature, which we review now.
	Among the following notions, the first one is most closely related to our \cref{def:os}.
	\begin{enumerate}[resume]
		\item Schreiner~\cite{SchreinerMatrixRegular} has considered \emph{matrix regular operator spaces},
			which are operator spaces where the hermitian part of each matrix level is a regularly ordered Banach space,
			and such that the conjugation maps~\eqref{eq:conjugation_map} are positive and the involution is isometric.
			In particular, a matrix regular operator space in Schreiner's sense is an operator system in our sense,
			but it is equipped with a full operator space structure satisfying Ruan's direct sum axiom.
		\item Werner~\cite{WernerSubspaces} has studied matrix ordered operator spaces and characterized, via a modified numerical radius,
			those which can be realized as $*$-invariant subspaces of some $\mathcal{B}(\mathcal{H})$ up to complete order isomorphism respecting the matrix norm topology.

			More recently, Kennedy, Kim, and Manor~\cite{KennedyKimManorNonunital} have used Werner's notion---reformulated via the canonical inclusion into the partial unitization being completely isometric---and proven a dual equivalence between the category of Werner's operator systems and a certain category of noncommutative convex sets.
		\item Russell~\cite{RussellOrderedSelfAdjoint} has combined the norm and positive cone into a \emph{gauge},
			which for unital operator systems are given by
			\[
				\gamma_n(x) \coloneqq \inf \, \{ t > 0 \mid x \le t \cdot 1_n \}.
			\]
			He has studied corresponding \emph{matrix gauge $*$-vector spaces},
			where one has an abstract gauge at each matrix level which does not necessarily arise from a norm and positive cone,
			and a coreflective subcategory of \emph{normal} matrix ordered $*$-operator spaces~\cite[Theorem~4.2]{RussellOrderedSelfAdjoint}.
			This gives a very broad categorical framework for nonunital operator systems.
			\item Recently, Blecher and Hay~\cite{BlecherHayBaseNormSpaces} have proposed a matrix-ordered generalization of base norm spaces as a framework for duals of operator systems.
				A complete complex noncommutative base norm space in their sense is an operator system in our sense:
				their base norm formula makes every Hermitian matrix level regularly ordered, scalar conjugation is positive and contractive, and their matrix norms satisfy Ruan's full direct sum axiom.
				
		\end{enumerate}
\end{remark}

Next, we will show how to obtain $\OpSys$ as a category of modules in our sense.
As is standard, we abbreviate cb = ``completely bounded'' and cp = ``completely positive''.

\begin{definition}
	Let $\mathsf{CMat}$ be the $\ROBan$-category having finite-dimensional complex matrix algebras $M_n$ as objects, and where the hom-objects $\mathsf{CMat}(M_n, M_m)$ are given by the spaces of $*$-preserving linear maps $M_n \to M_m$ equipped with the cb operator norm and the cone of cp maps.
\end{definition}

\begin{lem}
	$\mathsf{CMat}$ is indeed a $\ROBan$-category.
\end{lem}

\begin{proof}
	We start with the fact that the hom-objects are objects of $\ROBan$.
	The nontrivial part is to verify that the cb norm satisfies conditions~\ref{item:regularly_ordered_1} and~\ref{item:regularly_ordered_2}.
	For condition~\ref{item:regularly_ordered_1}, let $\Phi,\Psi \colon M_n \to M_m$ be $*$-preserving linear maps where $-\Psi \le \Phi \le \Psi$ in the cp order.
	We then need to prove that $\|\Phi\|_{\mathrm{cb}} \le \|\Psi\|_{\mathrm{cb}}$.
	Since every $M_k(\Psi)$ is positive, its operator norm is attained at the identity matrix~\cite[Corollary~2.9]{Paulsen}, so that
	\[
		\|\Psi\|_{\mathrm{cb}} = \sup_{k \in \N} \|M_k(\Psi)\| = \sup_{k \in \N} \|M_k(\Psi)(1_{kn})\| = \|\Psi(1_n)\|.
	\]
	In particular, for the desired $\|\Phi\|_{\mathrm{cb}} \le \|\Psi\|_{\mathrm{cb}}$,
	it is enough to prove the inequality
	\[
		\|M_k(\Phi)(X)\| \le \|M_k(\Psi)(1_{kn})\|
	\]
	for every $X \in M_k(M_n)$ with $\|X\| \le 1$ and every matrix level $k$.
	For self-adjoint $X$, we can write $X = X_+ - X_-$ with $X_+,X_- \ge 0$ and $X_+ + X_- = |X| \le 1_{kn}$.
	Then the inequalities $M_k(\Psi) \pm M_k(\Phi) \ge 0$ imply
	\[
		-M_k(\Psi)(1_{kn}) \le -M_k(\Psi)(|X|) \le M_k(\Phi)(X) \le M_k(\Psi)(|X|) \le M_k(\Psi)(1_{kn}).
	\]
	Hence indeed $\|M_k(\Phi)(X)\| \le \|M_k(\Psi)(1_{kn})\|$.
	For a general $X \in M_k(M_n)$ with $\|X\| \le 1$, consider its self-adjoint dilation
	\[
		\begin{pmatrix}
			0 & X \\
			X^* & 0
		\end{pmatrix}
		\in M_{2k}(M_n).
	\]
	Its square is $\begin{psmallmatrix} XX^* & 0 \\ 0 & X^*X \end{psmallmatrix}$, which has norm $\|X\|^2$, so the dilation itself has norm $\|X\|$.
	Also its image under $M_{2k}(\Phi)$ is
	\[
		\begin{pmatrix}
			0 & M_k(\Phi)(X) \\
			M_k(\Phi)(X)^* & 0
		\end{pmatrix},
	\]
	which likewise has norm $\|M_k(\Phi)(X)\|$.
	Applying the self-adjoint case at matrix level $2k$ therefore yields the desired bound, so that condition~\ref{item:regularly_ordered_1} holds.

	Condition~\ref{item:regularly_ordered_2} follows from Wittstock's decomposition theorem~\cite[Theorem~8.5]{Paulsen}: if $\|\Phi\|_{\mathrm{cb}} < 1$, then there exists a cp map $\Psi$ with $\|\Psi\|_{\mathrm{cb}} \le \|\Phi\|_{\mathrm{cb}} < 1$ and $\Psi \pm \Phi$ itself cp, that is, with $-\Psi \le \Phi \le \Psi$.
	Overall, each hom-object $\mathsf{CMat}(M_n,M_m)$ is therefore indeed a regularly ordered Banach space.

	It remains to consider composition.
	For fixed $\ell,m,n$, the usual composition of linear maps defines a bilinear map
	\[
		\mathsf{CMat}(M_m,M_\ell) \times \mathsf{CMat}(M_n,M_m) \longrightarrow \mathsf{CMat}(M_n,M_\ell).
	\]
	This bilinear map is positive, because the composite of cp maps is again cp.
	It is also contractive with respect to the cb norm, since
	\[
		\|\Phi \Psi\|_{\mathrm{cb}} \le \|\Phi\|_{\mathrm{cb}} \|\Psi\|_{\mathrm{cb}}
	\]
	for cb maps.
	Hence composition is a positive contractive bilinear morphism, so by the universal property of the projective tensor product it induces a composition morphism
	\[
		\mathsf{CMat}(M_m,M_\ell) \otimes_\pi^+ \mathsf{CMat}(M_n,M_m) \longrightarrow \mathsf{CMat}(M_n,M_\ell)
	\]
	in $\ROBan$.
	Also the unit map 
	\begin{align*}
		\R & \longrightarrow \mathsf{CMat}(M_n,M_n) \\
		1 & \longmapsto \id_{M_n}
	\end{align*}
	is a positive contraction.
	The associativity and unitality axioms for composition follow from the corresponding properties of composition of linear maps.
\end{proof}

\begin{theorem}
	\label{thm:op_sys_as_modules}
	There are equivalences of categories:
	\begin{enumerate}
		\item ${}_{\mathsf{CMat}} \Mod \simeq\mathsf{OpSys}$.
		\item $\Mod_\mathsf{CMat} \simeq \OpSys^\vee$.
	\end{enumerate}
\end{theorem}
\begin{proof}
	For the first claim, we construct a functor
	\[
		\mathsf{E} \: : \: \OpSys \longrightarrow {}_{\mathsf{CMat}} \Mod
	\]
	and then prove that it is an equivalence.

	Given an operator system $S$, we define a left $\mathsf{CMat}$-module $\mathsf{E}(S)$ by
	\[
		\mathsf{E}(S)(M_n) \coloneqq M_n(S)_h,
	\]
	where for a $*$-preserving linear map $\Psi \colon M_n \to M_m$, the action of $\Psi$ on $\mathsf{E}(S)$ is given by
	\begin{align*}
		M_n(S)_h & \longrightarrow M_m(S)_h \\
		x & \longmapsto \Psi(x),
	\end{align*}
	where $\Psi$ is applied in the obvious way that acts as identity on the $S$ tensor factor.
	The associativity and unitality axioms for a module follow from the corresponding properties of composition of linear maps.

	To finish the construction of $\mathsf{E}(S)$, we verify that the action maps
	\[
		\mathsf{CMat}(M_n,M_m) \otimes_\pi^+ M_n(S)_h \longrightarrow M_m(S)_h
	\]
	are indeed morphisms in $\ROBan$, or by \cref{roban_misc}\ref{roban_universal_property} equivalently that the bilinear action maps
	\[
		\mathsf{CMat}(M_n,M_m) \times M_n(S)_h \longrightarrow M_m(S)_h
	\]
	satisfy~\eqref{eq:roban_bilinear_morphism}.
	For this, first note that the weak direct sum axiom implies, by iteration and compression from a larger binary power if necessary, that
	\[
		\left\|
		\begin{pmatrix}
			x & & 0 \\
			  & \ddots & \\
			0 & & x
		\end{pmatrix}
		\right\|
		= \|x\|
	\]
	holds for every matrix size on the left.
	For positive $x$, the diagonal matrix is positive because it is a sum of corner embeddings obtained from scalar conjugation maps.
	If now $\Psi \colon M_n \to M_m$ is cp with $\|\Psi\|_{\mathrm{cb}} \le 1$,
	then the Kraus representation and the identity $\|\Psi(1_n)\| = \|\Psi\|_{\mathrm{cb}}$~\cite[Props.~3.6 and~4.7]{Paulsen} give a scalar matrix $\alpha$ with $\|\alpha\| \le 1$ such that
	\begin{equation}
		\label{eq:stinespring}
		\Psi(x)
		=
		\alpha^*
		\begin{pmatrix}
			x & & 0 \\
			  & \ddots & \\
			0 & & x
		\end{pmatrix}
		\alpha.
	\end{equation}
	Thus by \cref{def:os}, we can conclude that $x \ge 0$ and $\|x\| \le 1$ indeed implies $\Psi(x) \ge 0$ and $\|\Psi(x)\| \le 1$.
	Therefore $\mathsf{E}(S)$ is indeed a left $\mathsf{CMat}$-module.

	On morphisms, a morphism $f \colon S \to T$ of operator systems is sent to the module morphism whose component at $M_n$ is
	\[
		M_n(f)_h \: : \: M_n(S)_h \longrightarrow M_n(T)_h.
	\]
	These components are positive contractions by the definition of morphisms in $\OpSys$,
	and they commute with the $\mathsf{CMat}$-actions since scalar matrix maps commute with applying $f$ entrywise.
	The preservation of composition and identities is clear.
	This completes the construction of the functor $\mathsf{E}$.

	To show that $\mathsf{E}$ is an equivalence, we prove that it is fully faithful and essentially surjective.
	Faithfulness is immediate, since a $*$-preserving complex-linear map $f : S \to T$ is determined by its restriction to hermitian elements,
	which is the component of $\mathsf{E}(f)$ at $M_1$.
	For fullness, let $\eta \colon \mathsf{E}(S) \to \mathsf{E}(T)$ be a morphism of $\mathsf{CMat}$-modules.
	The component $\eta_{M_1} \colon S_h \to T_h$ extends uniquely to a $*$-preserving complex-linear map $f \colon S \to T$.
	For $s \in S_h$ and $a \in M_{n,h}$, we consider the element of $\mathsf{CMat}(M_1,M_n)$ determined by $1 \mapsto a$,
	the action of which is preserved by $\eta$, so that
	\[
		\eta_{M_n}(s \otimes a) = \eta_{M_1}(s) \otimes a
		\qquad
		\forall s \in S_h.
	\]
	Since such tensors span $M_n(S)_h$, we have $\eta_{M_n} = M_n(f)_h$ for every $n$.
	Each $\eta_{M_n}$ is a morphism in $\ROBan$, so $M_n(f)_h$ is positive and contractive for every $n$.
	Thus $f$ is a morphism of operator systems, and $\eta = \mathsf{E}(f)$ holds by construction.

	For the first claim, it thus remains to prove essential surjectivity.
	Let $\mathcal{F}$ be a left $\mathsf{CMat}$-module.
	After forgetting norms and positive cones, $\mathsf{CMat}$ becomes the full $\Vect_\R$-subcategory on the objects $(M_n)_h$.
	Since $(M_1)_h = \R$, \cref{ex:mor} and \cref{ex:vect_modules} identify the underlying vector spaces of $\mathcal{F}$ as
	\begin{align*}
		\mathcal{F}(M_n) & \cong \mathcal{F}(M_1)\otimes_\R \mathsf{CMat}(M_1, M_n) \\
						 & \cong \mathcal{F}(M_1)\otimes_\R M_{n,h} \\
						 & \cong (\mathcal{F}(M_1)_{\C} \otimes_{\C} M_n)_h.
	\end{align*}
	Let now $S \coloneqq \mathcal{F}(M_1)_\C$.
	We transport the regularly ordered Banach space structure of $\mathcal{F}(M_n)$ along these isomorphisms to $M_n(S)_h$.
	The scalar conjugation maps in \cref{def:os} are actions by cp contractions in $\mathsf{CMat}$,
	and hence are positive and contractive.
	The binary diagonal map $x \mapsto \begin{psmallmatrix} x & 0 \\ 0 & x \end{psmallmatrix}$ is also induced by a cp contraction, giving the inequality $\|\begin{psmallmatrix} x & 0 \\ 0 & x \end{psmallmatrix}\| \le \|x\|$.
	The reverse inequality follows by compressing back to either diagonal corner, again using a cp contraction.
	Thus $S$ is an operator system in the sense of \cref{def:os}.
	By construction, the displayed isomorphisms assemble to an isomorphism $\mathcal{F} \cong \mathsf{E}(S)$ of $\mathsf{CMat}$-modules.
	Hence $\mathsf{E}$ is essentially surjective, proving the first equivalence.

	The proof of the second claim is largely analogous.
	We construct a functor
	\[
		\mathsf{E}^{\vee} \: : \: \OpSys^\vee \longrightarrow \Mod_\mathsf{CMat}
	\]
	sending a dual operator system $T$ to the right $\mathsf{CMat}$-module $\mathsf{E}^{\vee}(T)$ with $\mathsf{E}^{\vee}(T)(M_n) \coloneqq M_n(T)_h$,
	with $\Psi \in \mathsf{CMat}(M_n,M_m)$ acting by
	\begin{align*}
		M_m(T)_h & \longrightarrow M_n(T)_h \\
		x & \longmapsto \Psi^\dagger(x),
	\end{align*}
	and where $\Psi^\dagger$ denotes the trace adjoint.\footnote{The trace adjoint of a linear map $\Psi \colon M_n \to M_m$
		is the unique linear map $\Psi^\dagger \colon M_m \to M_n$ such that $\mathrm{tr}(\Psi(a)b) = \mathrm{tr}(a\Psi^\dagger(b))$ for all $a \in M_n$ and $b \in M_m$,
		where $\mathrm{tr}$ denotes the unnormalized trace.
		It is contravariantly functorial in the sense that $(\Psi \circ \Phi)^\dagger = \Phi^\dagger \circ \Psi^\dagger$ and $\mathrm{id}^\dagger = \mathrm{id}$.}
	Functoriality of trace adjoints implies that the module axioms are satisfied.
	To check that the action maps are morphisms in $\ROBan$,
	let $\Psi \colon M_n \to M_m$ be cp with $\|\Psi\|_{\mathrm{cb}} \le 1$, and let $x \in M_m(T)_h$ be positive with $\|x\| \le 1$.
	Choose a Stinespring representation as in~\eqref{eq:stinespring}, with $\|\alpha\| \le 1$.
	Then, at the level of $T$-valued matrices, $\Psi^\dagger(x)$ is obtained by first applying the scalar conjugation $x \mapsto \alpha x \alpha^*$ and then taking the trace over the blocks.
	The scalar conjugation is positive and contractive by definition of dual operator system,
	and the block trace is so by iterated application of the dual weak direct sum axiom, padding by zero blocks if necessary.
	Thus $\Psi^\dagger(x)$ is again positive and of norm at most $1$, as required.

	While full faithfulness works as in the previous case,
	also the proof of essential surjectivity is similar.
	The only change is that a right $\mathsf{CMat}$-module $\mathcal{G}$ displays linear isomorphisms
	\begin{equation}
		\label{eq:dual_os_module_isomorphisms}
		\mathcal{G}(M_n)
		\cong
		\mathcal{G}(M_1)\otimes_\R \mathsf{CMat}(M_n,M_1)
		\cong
		(\mathcal{G}(M_1)_\C \otimes_\C M_n)_h,
	\end{equation}
	where the second isomorphism involves the canonical linear isomorphism $M_{n,h}^* \cong M_{n,h}$ obtained from the trace pairing.
	We now claim that transporting the cones and norms along these isomorphisms defines a dual operator system structure on $\mathcal{G}(M_1)_\C$.
	The point is that for every $\Psi \colon M_n \to M_m$ in $\mathsf{CMat}$, the diagram
	\[
		\begin{tikzcd}
			\mathsf{CMat}(M_m,M_1) \ar[r, "\cong"] \ar[d, "{-\circ\Psi}"'] & M_{m,h} \ar[d, "\Psi^\dagger"] \\
			\mathsf{CMat}(M_n,M_1) \ar[r, "\cong"] & M_{n,h}
		\end{tikzcd}
	\]
	commutes, as one can easily see from the trace pairing.
	Tensoring this diagram with $\mathcal{G}(M_1)$ shows that the above isomorphisms assemble to an isomorphism of right $\mathsf{CMat}$-modules once the right-hand side is equipped with the action by trace adjoints.

	It remains only to observe that the conditions required of a dual operator system hold.
	For $\alpha \in M_{n,m}$ with $\|\alpha\| \le 1$, the scalar conjugation map
	\begin{align*}
		M_n & \longrightarrow M_m \\
		x & \longmapsto \alpha^*x\alpha
	\end{align*}
	is the trace adjoint of the cp contraction
	\begin{align*}
		M_m & \longrightarrow M_n \\
		y & \longmapsto \alpha y\alpha^*.
	\end{align*}
	Moreover, the block trace map
	\begin{align*}
		M_{2n} & \longrightarrow M_n \\
		\begin{pmatrix}
			x & z \\
			z^* & y
		\end{pmatrix}
		& \longmapsto x+y
	\end{align*}
	is the trace adjoint of the cp contraction
	\begin{align*}
		M_n & \longrightarrow M_{2n} \\
		x & \longmapsto \begin{pmatrix}
			x & 0 \\
			0 & x
		\end{pmatrix}.
	\end{align*}
	Since the corresponding right module actions are morphisms in $\ROBan$, we conclude that these maps act positively and contractively on the matrix levels.
	Thus $\mathcal{G}(M_1)_\C$ is a dual operator system whose associated right $\mathsf{CMat}$-module is isomorphic to $\mathcal{G}$ by~\eqref{eq:dual_os_module_isomorphisms}.
	This proves essential surjectivity of the second equivalence.
\end{proof}

\begin{example}
	\label{ex:cmat_modules}
	Let us spell out the representable and corepresentable $\mathsf{CMat}$-modules.
	In particular, this will show that operator systems in our sense can be completely order isomorphic without being isomorphic in $\OpSys$.
	\begin{enumerate}
		\item
			For every $k \in \N$, the standard operator system structure on $M_k$ has matrix levels
			\[
				M_n(M_k)_h \cong M_{nk,h},
			\]
			which carry the usual positive semidefinite cones and operator norms.
			The representable left module $\mathsf{CMat}(M_k,-)$ gives an operator system with the same cones at all matrix levels, but with different norms for $k > 1$.
			Indeed, the Choi--Jamiolkowski isomorphism induced by the trace pairing~\cite[Theorem~3.14]{Paulsen} gives
			\[
				\mathsf{CMat}(M_k,M_n)
				\cong
				(M_n\otimes M_k)_h,
			\]
			with cp maps corresponding to positive semidefinite matrices.
			Also the module action by $\Psi \colon M_n \to M_m$ is postcomposition, matching the action of $\Psi$ on the right-hand side.
			However, the cb norm on the left-hand side does \emph{not} correspond to the usual operator norm on the right for $k > 1$.
			Already at $n = 1$, a positive $x \in M_k$ corresponds to the functional $y \mapsto \mathrm{tr}(xy)$ on $M_k$, whose cb norm is $\mathrm{tr}(x)$ rather than $\|x\|$.

			Overall, the systems $M_k$ and $\mathsf{CMat}(M_k,-)$ are completely order isomorphic, but for $k > 1$ not isomorphic in $\OpSys$.
		\item\label{item:cmat_left_modules_matrices}
			The corepresentable left module $\neg\mathsf{CMat}(-,M_k)$ behaves similarly.
			At matrix level $n$, its underlying vector space is $\mathsf{CMat}(M_n,M_k)^*$,
			and \cref{roban_internal_hom}\ref{roban_dual} shows that its positive cone and norm are given as follows.
			For every linear $\Omega \colon \mathsf{CMat}(M_n,M_k) \to \R$, we have
			\[
				\Omega \ge 0
				\qquad
				\Longleftrightarrow
				\qquad
				\Omega(\Psi) \ge 0 \quad \forall \, \textrm{cp } \Psi \colon M_n \to M_k,
			\]
			and
			\[
				\|\Omega\|
				=
				\inf_{-\Lambda \le \Omega \le \Lambda}
				\ \sup_{\Psi \in \mathsf{CMat}(M_n,M_k)_1}
				|\Lambda(\Psi)|,
			\]
			respectively.
			In particular, for positive $\Omega$ this is its ordinary dual norm.
			Again via the Choi--Jamiolkowski isomorphism and the trace pairing,
			the cone gets identified with the positive semidefinite cone in $M_{nk,h}$.
			However, for $k > 1$ the norm is again different from the operator norm on $M_{nk,h}$, for the same reason as in the previous item:
			already for $n=1$, transporting the norm along the isomorphism
			\[
				\mathsf{CMat}(M_1,M_k)^* \cong M_{k,h}
			\]
			yields the trace norm on the right-hand side.
		\item
			Overall, we obtain complete order isomorphisms
			\[
				\mathsf{CMat}(M_k, -)
				\cong
				(- \otimes M_k)_h
				\cong
				\neg\mathsf{CMat}(-, M_k),
			\]
			but neither of these two isomorphisms nor their composite is completely isometric if $k > 1$.
			For the composite, the norms coincide at level $1$, as they both correspond to the trace norm on $M_{k,h}$,
			but they differ at level $k$ as follows.
			The identity map in $\mathsf{CMat}(M_k,M_k)$ has cb norm $1$ in the representable module on the left.
			Its counterpart in the corepresentable module on the right is the ordinary trace functional $\operatorname{Tr}$ on linear endomorphisms of $M_{k,h}$.
			This functional has norm $k^2$:
			if $\|\Psi\|_{\mathrm{cb}} \le 1$, then in particular $\|\Psi\| \le 1$ in operator norm on $M_{k,h}$.
			Thus every complex eigenvalue of $\Psi$ has modulus $\le 1$, and hence $|\operatorname{Tr}(\Psi)| \le k^2$.
			Equality is attained at $\Psi = \id_{M_k}$, which has cb norm $1$.
		\item
			\label{item:cmat_right_modules_matrices}
			We can also consider $M_k$ as a dual operator system, with the usual positive cones on matrix levels, but with the trace norm on each matrix level $M_n(M_k)_h \cong M_{nk,h}$.
			Again the representable and corepresentable right modules $\mathsf{CMat}(-,M_k)$ and $\neg\mathsf{CMat}(M_k,-)$ have the same cones at all matrix levels, but different norms for $k > 1$.
			We refrain from spelling out the details here, since they are analogous to the left module case.
	\end{enumerate}
\end{example}

We now instantiate our general separation and representation theorems in the case of $\mathsf{CMat}$-modules,
formulated via \cref{thm:op_sys_as_modules} in terms of operator systems but proven in terms of $\mathsf{CMat}$-modules.
In order to stay as close as possible to the standard results on unital operator systems,
we formulate these results after applying the complete order isomorphisms of \cref{ex:cmat_modules},
even though these are not completely isometric.

\begin{corollary}
	\label{cor:op_sys_separation}
	Completely positive maps into matrix algebras distinguish morphisms of operator systems.
\end{corollary}
\begin{proof}
	Let $\mathcal{F}$ and $\mathcal{G}$ be $\mathsf{CMat}$-modules, and let $\alpha, \beta \colon \mathcal{F} \to \mathcal{G}$
	be morphisms.
	If $\alpha \neq \beta$, then by \cref{separation} there exist $k \in \N$ and a morphism $\varphi \colon \mathcal{G} \to \mathsf{CMat}(M_k, -)^*$
	such that $\varphi \alpha \neq \varphi \beta$.
	By the previous example, the corepresentable module $\mathsf{CMat}(M_k, -)^*$ is completely order isomorphic to the usual matrix ordered vector space $M_k$,
	so that composing with this isomorphism proves the desired separation.
\end{proof}

\begin{corollary}
	\label{cor:op_sys_representation}
	Every operator system has a complete order embedding into a product of matrix algebras.
\end{corollary}
\begin{proof}
	Let $\mathcal{F}$ be a left $\mathsf{CMat}$-module.
	By \Cref{representation}, there is a composite of regular monomorphisms
	\[
		\mathcal{F}
		\longrightarrow
		\prod_{M_k \in \mathsf{CMat}}
		\prod_{\varphi \in \ROBan(\mathcal{F}(M_k),\R)}
		\mathsf{CMat}(M_k,-)^*.
	\]
	By \cref{ex:cmat_modules}\ref{item:cmat_left_modules_matrices}, each corepresentable $\mathsf{CMat}(M_k,-)^*$ is just the matrix algebra $M_k$ with its usual matrix ordering (but with different matrix norms).
	Thanks to~\Cref{lem:module_neg_monos,roban_regular_embeddings}, the composite of regular monomorphisms is an order embedding, so that $\mathcal{F}$ embeds into a product of matrix algebras.
\end{proof}

\begin{remark}
	\label{rem:op_sys_representation}
	Since the product of matrix algebras in~\cref{cor:op_sys_representation} comes from a product in $\ROBan$,
	the elements of this product are \emph{uniformly bounded} families with respect to the corepresentable norms.
	These become the trace norms under the complete order isomorphisms of \cref{ex:cmat_modules} at matrix level $1$.
	This is stronger than uniform boundedness in operator norm, since
	$\|x\|_{\mathrm{op}} \le \|x\|_1 \le k\|x\|_{\mathrm{op}}$ for $x\in M_k$.

	In fact, for a unital operator system $S$,
	the proof of \cref{cor:op_sys_representation} recovers the usual representation theorem~\cite[Theorem~13.1]{Paulsen} after a normalization.
	To see this, write $\Psi_i \colon S \to M_{k_i}$ for the cp components of the embedding, indexed by $i \in J$.
	By \cref{ex:cmat_modules}\ref{item:cmat_left_modules_matrices}, the relevant norm on $M_{k_i}$ is the trace norm,
	so that contractivity of $\Psi_i$ yields in particular
	\[
		\operatorname{tr}\bigl(\Psi_i(1_S)\bigr) \le 1.
	\]
	Hence, by the norm formula for cp maps~\cite[Proposition~3.6]{Paulsen}, we have $\|\Psi_i\|_{\mathrm{cb}}=\|\Psi_i(1_S)\|_{\mathrm{op}}\le 1$.
	Therefore the complete order embedding 
	\begin{equation}
		\label{eq:S_embed}
		S \longrightarrow \prod_{i \in J} M_{k_i}
	\end{equation}
	is additionally completely contractive as a map from $S$ into the usual uniformly bounded product of matrix algebras (equipped with operator norm).

	However, the embedding~\eqref{eq:S_embed} does not yet have to be unital.
	To fix this, let $p_i$ be the support projection of $a_i\coloneqq\Psi_i(1_S)$.
	By the standard support property for cp maps, the range of $\Psi_i$ lies in $p_iM_{k_i}p_i$.
	After identifying each corner with a full matrix algebra, we may therefore assume that every $a_i$ is invertible.
	Then the map
	\begin{align*}
		S & \longrightarrow \prod_{i \in J} M_{k_i} \\
		s & \longmapsto (a_i^{-1/2}\Psi_i(s)a_i^{-1/2})_{i \in J}
	\end{align*}
	has ucp components and therefore lands in the uniformly bounded product.
	Passing to the support corner and conjugating by $a_i^{-1/2}$ preserve and reflect positivity at every matrix level, so this is still a complete order embedding.
	This yields the usual representation in $\mathcal{B}(\mathcal{H})$ upon taking the block-diagonal representation of the product.
\end{remark}

\begin{remark}
	\label{rem:op_sys_extension}
	One might also hope that \cref{extension} might produce a version of Arveson's extension theorem for operator systems.
	Unfortunately, its hypothesis does not hold in the present setting, since $\R$ is not regularly injective in $\ROBan$!
	It thus remains open to find a general extension theorem for modules which would apply to operator systems.

	Indeed, Min's~\cite[Section~5.3, Counterexample~(2)]{MinExponentialLaw}
	gives an equalizer $e\colon E\to V$ in $\ROBan$ together with an element $x\in E$ such that
	\[
		N_E(x)>N_V(e(x)),
	\]
	where we use the half-norms of \cref{roban_half_norm}.
	It follows that there is a positive contraction $\varphi\colon E\to\R$ with $\varphi(x)>N_V(e(x))$.
	Such a morphism cannot extend along $e$, since every positive contraction $\widetilde{\varphi}\colon V\to\R$ satisfies
	\[
		\widetilde{\varphi}(e(x))\le N_V(e(x)),
	\]
	again by \cref{roban_half_norm}.
\end{remark}

\begin{remark}
	\label{rem:os_tensor_closed}
	The category $\mathsf{CMat}$ has an evident symmetric monoidal $\ROBan$-category structure.
	On objects, the tensor product and monoidal unit are given by
	\[
		M_m\boxtimes M_n
		\coloneqq
		M_m\otimes M_n
		\cong
		M_{mn},
		\qquad
		E\coloneqq M_1,
	\]
	while on morphisms $\Phi\boxtimes\Psi\coloneqq\Phi\otimes\Psi$.
	By the fact that the tensor product of cp contractions is again a cp contraction,
	\cref{roban_misc}\ref{roban_universal_property} shows that this operation on morphisms is itself a morphism in $\ROBan$.
	The associators, unitors, and braidings are the obvious ones.
	Hence this makes $\mathsf{CMat}$ indeed into a symmetric monoidal $\ROBan$-category.

	We can now apply \cref{rem:day_convolution} in order to obtain a closed symmetric monoidal structure on ${}_{\mathsf{CMat}}\Mod \simeq \OpSys$.
	The monoidal unit object is the left module $\mathsf{CMat}(E,-)$, corresponding to $\C$ with its obvious operator system structure.
	Given that the universal property of the Day convolution tensor looks quite similar to the universal property of the maximal tensor product of unital operator systems~\cite[Thm.~5.8]{KavrukPaulsenTodorovTomfordeTensor},
	we expect that the resulting symmetric monoidal structure on $\OpSys$ will be closely related to the maximal tensor of~\cite{KavrukPaulsenTodorovTomfordeTensor}.

	Moreover, we expect that the resulting closed structure on $\OpSys$ will supply the desired notion of dual operator system hinted at in \cref{rem:os_comments_2}\ref{rem:os_wrong_dual}.
	Indeed every left module $\mathcal{F}$ will have an internal-hom dual $\mathcal{F}\Rightarrow\mathsf{CMat}(M_1,-)$, which thanks to the formula~\eqref{eq:day_convolution_internal_hom} is given by
	\begin{align*}
		(\mathcal{F}\Rightarrow\mathsf{CMat}(M_1,-))(M_n)
		&\:=\:
		\int_{M_m\in\mathsf{CMat}}
		\bigl(\mathcal{F}(M_m) \Rightarrow 
		\mathsf{CMat}(M_1, M_n\otimes M_m)\bigr) \\[3pt]
		&\:\cong\:
		[\mathsf{CMat},\ROBan]\bigl(\mathcal{F},\mathsf{CMat}(M_1, M_n\otimes-)\bigr).
	\end{align*}
	The Choi--Jamiolkowski isomorphism~\cite[Theorem~3.14]{Paulsen} implements a complete order isomorphism $\mathsf{CMat}(M_1,M_n\otimes -) \cong \mathsf{CMat}(M_n,-)$.\footnote{This isomorphism is also completely bounded: For fixed $n$, the elementary Choi formulas show that the norm of the Choi map
		$\mathsf{CMat}(M_n,M_m)\to\mathsf{CMat}(M_1,M_n\otimes M_m)$ and that of its inverse are both at most $n^2$, independently of $m$.
		Hence postcomposition with these maps defines mutually inverse bounded maps between the two hom-objects.}
	Therefore we obtain a complete order isomorphism
	\[
		(\mathcal{F}\Rightarrow\mathsf{CMat}(M_1,-))(M_n)
		\cong
		[\mathsf{CMat},\ROBan]\bigl(\mathcal{F},\mathsf{CMat}(M_n,-)\bigr),
	\]
	which however is generally \emph{not} a complete isometry.
	For an operator system $S$, this means that the positive cone of $S\Rightarrow\C$ at matrix level $n$ 
	is the cone of cp maps $S\to M_n$, as one would expect.
	In particular, this ``dual'' operator system $S\Rightarrow\C$ is importantly distinct from taking the levelwise dual:
	the latter is a dual operator system in the sense of \cref{def:dual_os}, i.e.~an object of $\OpSys^\vee$.
\end{remark}

\appendix
\section{Appendix}
\label{sec:appendix}
\subsection{Auxiliary material on limits and colimits}
\label{app:limits_colimits}

In this appendix, $\cV$ is a B\'enabou cosmos.
We recall some basic facts about limits and colimits in $\cV_0$ and in categories of modules.

\begin{lem}[{cf.~\cite[\S\S~3.3 and~3.8]{Kelly}}]
	\label{lem:module_pointwise_conical}
	For any small $\cV$-category $\cC$, conical limits and colimits in $\Mod_\cC$ exist and are computed pointwise.
\end{lem}

\begin{proof}
	Since every small limit can be constructed from products and equalizers, and every small colimit from coproducts and coequalizers, it is enough to treat these four cases.

	For a family of right $\cC$-modules $(\mathcal{F}_i)_{i \in I}$, define
	\[
		\Bigl(\prod_{i \in I} \mathcal{F}_i\Bigr)(X) \coloneqq \prod_{i \in I} \mathcal{F}_i(X).
	\]
	Then for every pair of objects $X,Y \in \cC$, there is a unique morphism
	\[
		\Bigl(\prod_{i \in I} \mathcal{F}_i\Bigr)(Y) \otimes \cC(X,Y)
		\longrightarrow
		\Bigl(\prod_{i \in I} \mathcal{F}_i\Bigr)(X)
	\]
	whose $i$-th component is the composite
	\[
		\begin{tikzcd}[column sep=large]
			\Bigl(\prod_{j \in I} \mathcal{F}_j(Y)\Bigr) \otimes \cC(X,Y) \ar[r, "\pi_i \otimes \id"] & \mathcal{F}_i(Y) \otimes \cC(X,Y) \ar[r, "(\mathcal{F}_i)_{X,Y}"] & \mathcal{F}_i(X)
		\end{tikzcd}
	\]
	The module axioms follow because the product projections are jointly monic, and the universal property of the product in $\Mod_\cC$ is immediate from the pointwise universal property in $\cV_0$.

	Now let $\alpha,\beta \colon \mathcal{F} \to \mathcal{G}$ be morphisms of right $\cC$-modules, and consider the pointwise equalizers
	\[
		\begin{tikzcd}
			\mathcal{E}(X) \ar[r, "e_X"] & \mathcal{F}(X) \ar[r, shift left=1.5, "\alpha_X"] \ar[r, shift right=1.5, "\beta_X"'] & \mathcal{G}(X)
		\end{tikzcd}
	\]
	Then for every pair of objects $X,Y \in \cC$, the morphism of module square for $\alpha$ and $\beta$ gives the commutative diagram
	\[
		\begin{tikzcd}[column sep=large, row sep=large]
			\mathcal{E}(Y) \otimes \cC(X,Y)
				\ar[d, dashed, "\mathcal{E}_{X,Y}"']
				\ar[r, "e_Y \otimes \id"]
			&
			\mathcal{F}(Y) \otimes \cC(X,Y)
				\ar[d, "\mathcal{F}_{X,Y}"']
				\ar[r, shift left=.8ex, "\beta_Y \otimes \id"]
				\ar[r, shift right=.8ex, "\alpha_Y \otimes \id"']
			&
			\mathcal{G}(Y) \otimes \cC(X,Y)
					\ar[d, "\mathcal{G}_{X,Y}"']
			\\
			\mathcal{E}(X)
				\ar[r, "e_X"']
			&
			\mathcal{F}(X) 
				\ar[r, shift right=.8ex, "\alpha_X"']
				\ar[r, shift left=.8ex, "\beta_X"]
			&
			\mathcal{G}(X)
		\end{tikzcd}
	\]
	where the dashed arrow is induced.
	Again, the module axioms and the universal property of the equalizer follow pointwise because the equalizer morphisms $e_X$ are monic,
	and this very diagram shows that they are the components of a morphism of right $\cC$-modules.
	Its universal property is checked by pointwise verification after postcomposing with the monomorphisms $e_X$.

	For coproducts, define
	\[
		\Bigl(\coprod_{i \in I} \mathcal{F}_i\Bigr)(X) \coloneqq \coprod_{i \in I} \mathcal{F}_i(X).
	\]
	Since $\cV$ is symmetric monoidal closed, the functor $- \otimes \cC(X,Y)$ is a left adjoint and therefore preserves coproducts.
	Hence for every pair of objects $X,Y \in \cC$ we have a canonical isomorphism
	\[
		\Bigl(\coprod_{i \in I} \mathcal{F}_i(Y)\Bigr) \otimes \cC(X,Y)
		\,\cong\,
		\coprod_{i \in I} \bigl(\mathcal{F}_i(Y) \otimes \cC(X,Y)\bigr),
	\]
	and therefore a unique action morphism
	\[
		\Bigl(\coprod_{i \in I} \mathcal{F}_i\Bigr)(Y) \otimes \cC(X,Y)
		\longrightarrow
		\Bigl(\coprod_{i \in I} \mathcal{F}_i\Bigr)(X)
	\]
	whose restriction to the $i$-th summand is the action of $\mathcal{F}_i$.
	The module axioms and the universal property are then once again checked pointwise, using that the coproduct injections are jointly epimorphic.

	Finally, let $\alpha,\beta \colon \mathcal{F} \to \mathcal{G}$ be morphisms of right $\cC$-modules, and consider the pointwise coequalizers
	\[
		\begin{tikzcd}
			\mathcal{F}(X) \ar[r, shift left=1.5, "\alpha_X"] \ar[r, shift right=1.5, "\beta_X"'] & \mathcal{G}(X) \ar[r, "q_X"] & \mathcal{Q}(X)
		\end{tikzcd}
	\]
	For every pair of objects $X,Y \in \cC$, the functor $- \otimes \cC(X,Y)$ preserves coequalizers, so we have another coequalizer
	\[
		\begin{tikzcd}[column sep=large]
			\mathcal{F}(Y) \otimes \cC(X,Y) \ar[r, shift left=1.5, "\alpha_Y \otimes \id"] \ar[r, shift right=1.5, "\beta_Y \otimes \id"'] & \mathcal{G}(Y) \otimes \cC(X,Y) \ar[r, "q_Y \otimes \id"] & \mathcal{Q}(Y) \otimes \cC(X,Y).
		\end{tikzcd}
	\]
	On the other hand, naturality gives the commutative diagram
	\[
		\begin{tikzcd}[column sep=large, row sep=large]
			\mathcal{F}(Y) \otimes \cC(X,Y)
				\ar[d, "\mathcal{F}_{X,Y}"']
				\ar[r, shift left=.8ex, "\alpha_Y \otimes \id"]
				\ar[r, shift right=.8ex, "\beta_Y \otimes \id"']
			&
			\mathcal{G}(Y) \otimes \cC(X,Y)
				\ar[d, "\mathcal{G}_{X,Y}"']
				\ar[r, "q_Y \otimes \id"]
			&
			\mathcal{Q}(Y) \otimes \cC(X,Y)
				\ar[d, dashed, "\mathcal{Q}_{X,Y}"']
			\\
			\mathcal{F}(X)
				\ar[r, shift left=.8ex, "\alpha_X"]
				\ar[r, shift right=.8ex, "\beta_X"']
			&
			\mathcal{G}(X)
				\ar[r, "q_X"]
			&
			\mathcal{Q}(X)
		\end{tikzcd}
	\]
	where the dashed arrow is induced.
	The module axioms follow from those of $\mathcal{F}$ and $\mathcal{G}$ after precomposing with the relevant tensor products of the coequalizer maps $q_X$, using that tensoring preserves coequalizers.
	If $\eta \colon \mathcal{G} \to \mathcal{H}$ coequalizes $\alpha$ and $\beta$, then each component $\eta_X$ factors uniquely through $q_X$.
	These factorizations form a module morphism by pointwise verification after precomposing with the epimorphisms $q_Y \otimes \id$.
	Hence $\mathcal{Q}$ is the coequalizer of $\alpha$ and $\beta$ in $\Mod_\cC$.

	Thus products, equalizers, coproducts, and coequalizers are all computed pointwise in $\Mod_\cC$.
	The claimed statement on small conical limits and colimits follows.
\end{proof}

We first recall enriched coends and then turn to enriched ends.
For both, we take $\cJ$ to be a small $\cV$-category and $H \colon \cJ^\op \otimes \cJ \longrightarrow \cV$ a $\cV$-functor.

\begin{definition}[{cf.~\cite[\S 2.1]{Kelly}}]
	\label{def:enriched_coend}
	An \textbf{enriched coend} 
	\[
		\int^{X \in \cJ} H(X,X)
	\]
	is a colimit of the diagram
	\begin{equation}
		\label{cowedge_condition}
		\begin{tikzcd}[column sep=large, row sep=large]
			\vdots & \vdots \\[-1.5em]
			\vdots &
			H(Z,Z)
			&
			\\
			\cJ(Y,Z) \otimes H(Z,Y)
				\ar[ur]
				\ar[dr]
			& \vdots
			\\
			\vdots &
			H(Y,Y)
			&
			\\[-1.5em]
			\vdots & \vdots
		\end{tikzcd}
	\end{equation}
	where the two unnamed families of morphisms are induced by the functoriality of $H$ in the second and first variable, respectively.
\end{definition}

A cocone over the diagram to an object $A$ is also called an \textbf{enriched cowedge} from $H$ to $A$.
Thus an enriched coend consists of an object $\int^{X \in \cJ} H(X,X)$ together with an enriched cowedge from $H$ to it, which is initial among all enriched cowedges.
It can also be computed as the coequalizer
\[
	\begin{tikzcd}[column sep=large]
		\coprod_{Y,Z \in \cJ} \bigl( \cJ(Y,Z) \otimes H(Z,Y)\bigr)
		\ar[r, shift left=.8ex]
		\ar[r, shift right=.8ex]
		&
		\coprod_{X \in \cJ} H(X,X)
		\ar[r]
		&
		\int^{X \in \cJ} H(X,X)
	\end{tikzcd}
\]
where the two parallel morphisms are induced by the functoriality of $H$ in the second and first variable, respectively.

We also use the notion of enriched end, which is dual to enriched coend, but with tensors replaced by internal homs.
Concretely, the \textbf{enriched end}
\[
	\int_{X \in \cJ} H(X,X)
\]
is given by the limit of the diagram
\begin{equation}
	\label{wedge_condition}
	\begin{tikzcd}[column sep=large, row sep=large]
		\vdots & \vdots \\[-1.5em]
		H(Y,Y) \ar[dr] & \vdots \\
		\vdots & \cJ(Y,Z) \Rightarrow H(Y,Z) \\
		H(Z,Z) \ar[ur] & \vdots \\[-1.5em]
		\vdots & \vdots
	\end{tikzcd}
\end{equation}
where the two morphisms into $\cJ(Y,Z) \Rightarrow H(Y,Z)$ are the transposes of the functoriality morphisms of $H$ under the hom-tensor adjunction.
A general cone is also called an \textbf{enriched wedge}, and thus an enriched end consists of an object $\int_{X \in \cJ} H(X,X)$ together with a terminal enriched wedge from it to $H$.
It can also be computed as the equalizer
\begin{equation}
	\label{eq:end_equalizer}
	\begin{tikzcd}[column sep=large]
		\int_{X \in \cJ} H(X,X)
		\ar[r]
				&
				\prod_{X \in \cJ} H(X,X)
				\ar[r, shift left=.8ex]
				\ar[r, shift right=.8ex]
				&
				\prod_{Y,Z \in \cJ} \bigl(\cJ(Y,Z) \Rightarrow H(Y,Z)\bigr)
	\end{tikzcd}
\end{equation}
where the $(Y,Z)$-components of the two parallel morphisms are, respectively, the obvious product projections followed by transposes of the functoriality morphisms of $H$ under the hom-tensor adjunction.

\newpage\thispagestyle{empty}
\addtocontents{toc}{\protect\vspace{-0.75em}}
\addcontentsline{toc}{part}{Bibliography} 
{\linespread{1}\bibliographystyle{dpbib} \bibliography{references}}

\end{document}